\documentclass[preprint,1p]{elsarticle}

\usepackage{amsfonts}
\usepackage{amsmath}
\usepackage{empheq} 
\usepackage{amsthm}
\usepackage[utf8]{inputenc}
\usepackage[english]{babel}
\usepackage{epstopdf}
\usepackage{bmpsize}
\usepackage{epsfig}
\usepackage{hyperref}
\usepackage[english]{varioref}
\usepackage{amssymb}
\usepackage{multicol}
\usepackage{dcolumn}
\usepackage{geometry}
\usepackage{fancyhdr}
\usepackage[mathcal]{eucal}
\usepackage{mathrsfs}
\usepackage{color, colortbl}
\usepackage{microtype}
\usepackage{longtable}
\usepackage[toc,page]{appendix}
\usepackage{bm}
\usepackage{pifont}
\usepackage{fleqn}
\usepackage{graphicx}
\usepackage{txfonts}

\allowdisplaybreaks
\DisableLigatures{encoding = *, family = * }

\numberwithin{equation}{section}
\newcommand{\norm}[2]{{\left\|#1\right\|}_{#2}}

\newcommand{\intr}[1]{\underset{#1}{\int}}
\newcommand{\Do}[1]{D_{#1}}

\newcommand{\RR}{\mathbb{R}}

\newtheorem{theorem}{Theorem}[section]
\newtheorem{proposition}{Proposition}[section]
\newtheorem{lemma}{Lemma}[section]

\title{Null controllability for a heat equation with a singular inverse-square potential involving the distance to the boundary function}

\author[rvt]{Umberto Biccari \fnref{fn1}}
\ead{ubiccari@bcamath.org - u.biccari@gmail.com}

\author[rvt1]{Enrique Zuazua \fnref{fn2}}
\ead{enrique.zuazua@uam.es}

\fntext[fn1]{The work of this author was supported by the Grant FA9550-14-1-0214 of the EOARD-AFOSR, the MTM2014-52347 and SEV-2013-0323 Grants of the MINECO and the BERC 2014-2017 program of the Basque Government}
\fntext[fn2]{The work of this author was supported by the Grants FA9550-14-1-0214 of the EOARD-AFOSR, FA9550-15-1-0027 of AFOSR and MTM2014-52347 of the MINECO,  and  by a Humboldt Research Award at the University of Erlangen-N\"urnberg.}
\address[rvt]{BCAM - Basque Center for Applied Mathematics, Alameda de Mazarredo 14, 48009, Bilbao, Basque Country, Spain - ubiccari@bcamath.org - u.biccari@gmail.com} 
\address[rvt1]{Universidad Autonoma de Madrid - Departamento de Matemáticas, Campus de Cantoblanco, 28049, Madrid, Spain - enrique.zuazua@uam.es}

\begin{document}

\begin{abstract}
This article is devoted to the analysis of control properties for a heat equation with singular potential $\mu/\delta^2$, defined on a bounded $C^2$ domain $\Omega\subset\RR^N$, where $\delta$ is the distance to the boundary function. More precisely, we show that for any $\mu\leq 1/4$ the system is exactly null controllable using a distributed control located in any open subset of $\Omega$, while for $\mu>1/4$ there is no way of preventing the solutions of the equation from blowing-up. The result is obtained applying a new Carleman estimate. 
\end{abstract}


\maketitle

%section 1 
\section{Introduction and main results}\label{introduction}
Let $T>0$ and set $Q:=\Omega\times(0,T)$, where $\Omega\subset\RR^N$, $N\geq 3$, is a bounded and $C^2$ domain, and let $\Gamma:=\partial\Omega$. Moreover, let $\delta(x):=\textrm{dist}(x,\partial\Omega)$ be the distance to the boundary function. We are interested in proving the exact null controllability for a heat equation with singular inverse-square potential of the type $-\mu/\delta^2$, that is, given the operator 
\begin{align}\label{singular op}
\mathcal{A}=\mathcal{A}(\mu):=-\Delta-\frac{\mu}{\delta^2}\mathcal{I},\;\;\;\mu\in\RR,
\end{align} 
where $\mathcal{I}$ indicates the identical operator, we are going to consider the following parabolic equation
\begin{align}\label{heat hardy nh}
	\renewcommand*{\arraystretch}{1.3}
	\left\{\begin{array}{ll}
		\displaystyle u_t-\Delta u-\frac{\mu}{\delta^2}u=f, & (x,t)\in Q
		\\ u=0, & (x,t)\in\Gamma\times(0,T)
		\\ u(x,0)=u_0(x), & x\in\Omega,
	\end{array}\right.
\end{align}
with the intent of proving that it is possible to choose the control function $f$ in an appropriate functional space $X$ such that the corresponding solution of \eqref{heat hardy nh} satisfies 
\begin{align}\label{control result}
	u(x,T)=0,\;\;\textrm{ for all }x\in\Omega.
\end{align}
In particular, the main result of this paper will be the following.
\begin{theorem}\label{control thm}
	Let $\Omega\subset\RR^N$ be a bounded $C^2$ domain and assume $\mu\leq 1/4$. Given any non-empty open set $\omega\subset\Omega$, for any time $T>0$ and any initial datum $u_0\in L^2(\Omega)$, there exists a control function $f\in L^2(\omega\times(0,T))$ such that the solution of \eqref{heat hardy nh} satisfies \eqref{control result}.
\end{theorem}
\indent The upper bound for the coefficient $\mu$ is related to a generalisation of the classical Hardy-Poincar\'e presented in \cite{brezis1997hardy} and plays a fundamental role in our analysis. Indeed, in \cite{cabre1999existence} is shown that, for $\mu>1/4$, \eqref{heat hardy nh} admits no positive weak solution for any $u_0$ positive and $f=0$. Moreover, there is instantaneous and complete blow-up of approximate solutions. 
\\ 
\\
\indent As it is by now classical, for proving Theorem \ref{control thm} we will apply the Hilbert Uniqueness Method (HUM, \cite{lions1988controlabilite}); hence the controllability property will be equivalent to the observability of the adjoint system associated to \eqref{heat hardy nh}, namely 
\begin{align}\label{heat hardy adj}
	\renewcommand*{\arraystretch}{1.3}	
	\left\{\begin{array}{ll}
		\displaystyle v_t+\Delta v+\frac{\mu}{\delta^2}v=0, & (x,t)\in Q
		\\ v=0, & (x,t)\in\Gamma\times (0,T)
		\\ v(x,T)=v_T(x), & x\in\Omega.
	\end{array}\right.
\end{align}
\indent More in details, for any $\mu\leq 1/4$ we are going to prove that there exists a positive constant $C_T$ such that, for all $v_T\in L^2(\Omega)$, the solution of \eqref{heat hardy adj} satisfies 
\begin{align}\label{observ ineq}
	\intr{\Omega} v(x,0)^2\,dx\leq C_T\intr{\omega\times(0,T)}v(x,t)^2\,dxdt.
\end{align} 
\\
\indent The inequality above, in turn, will be obtained as a consequence of a Carleman estimate for the solution of \eqref{heat hardy adj}, which is derived taking inspiration from the works \cite{cazacu2014controllability} and \cite{ervedoza2008control}.
\\
\indent Furthermore, the bound $\mu\leq 1/4$ is sharp for our controllability result, as we are going to show later in this work.
\\
\\
\indent Singular inverse-square potentials arise in quantum cosmology (\cite{berestycki1997existence}), in electron capture problems (\cite{giri2008electron}), but also in the linearisation of reaction-diffusion problems involving the heat equation with supercritical reaction term (\cite{azorero1998hardy}); also for these reasons, evolution problems involving this kind of potentials have been intensively studied in the last decades. 
\\
\indent In the pioneering work of 1984 \cite{baras1984heat}, Baras and Goldstein considered a heat equation in a bounded domain $\Omega\subset\RR^N$, for $N\geq 3$, with potential $-\mu/|x|^2$ and positive initial data, and proved that the Cauchy problem is well posed in the case $\mu\leq\mu^*:=(N-2)^2/4$, while it has no solution if $\mu>\mu^*$. We remind here that $\mu^*$ is the critical value for the constant in the Hardy inequality, guaranteeing that, for any $u\in H_0^1(\Omega)$, it holds 
\begin{align}\label{hardy_classical}
	\int_{\Omega}|\nabla u|^2\,dx\geq\mu^*\int_{\Omega}\frac{u^2}{|x|^2}\,dx.
\end{align}
$\indent$ The result by Baras and Goldstein was, in our knowledge, the first on the topic and it has later been improved by Vazquez and Zuazua in \cite{vazquez2000hardy}. There the authors present a complete description of the functional framework in which it is possible to obtain well-posedness for the singular heat equation they analyse; in particular, they prove that when $\mu<\mu^*$ the corresponding operator generates a coercive quadratic form form in $H_0^1(\Omega)$ and this allows to show well-posedness in the classical variational setting. On the contrary, when $\mu=\mu^*$, the space $H_0^1(\Omega)$ has to be slightly enlarged, due to the logarithmic singularity of the solutions at $x=0$. 
\\
\indent Also the question of whether it is possible to control heat equations involving singular inverse-square potentials has already been addressed in the past, and there is nowadays an extended literature on this topic. 
\\
\indent Among other works, we remind here the one by Ervedoza, \cite{ervedoza2008control}, and the one by Vancostenoble and Zuazua, \cite{vancostenoble2008null}. In both, the authors consider the case of an equation defined on a smooth domain containing the origin and prove exact null controllability choosing a control region inside of the domain, away from the singularity point $x=0$. 
\\
\indent In particular, in \cite{vancostenoble2008null} the null controllability result is obtained choosing a control region containing an annular set around the singularity and using appropriate cut-off functions in order to split the problem in two:
\begin{itemize}
	\item[•] in a region of the domain away from the singularity, in which it is possible to employ classical Carleman estimates;
	\item[•] in the remaining part of the domain, a ball centred in the singularity, in which the authors can apply polar coordinates and reduce themselves to a one-dimensional equation, which is easier to handle.
\end{itemize} 
$\indent$ In \cite{ervedoza2008control}, instead, the author generalises the result by Vancostenoble and Zuazua, proving controllability from any open subset of $\Omega$ that does not contains the singularity. This result is obtained deriving a new Carleman estimate, involving a weight that permits to avoid the splitting argument introduced in is \cite{vancostenoble2008null}.
\\
\indent Finally, it is worth to mention also the work \cite{cazacu2014controllability}, by Cazacu. In this paper, it is treated the case of a potential with singularity located on the boundary of the domain and it is proved again null controllability with an internal control. Also this result follows from a new Carleman estimate that is derived using the same kind of weight function proposed by Ervedoza, but with some suitable modifications that permit to deal with the case of boundary singularities. Moreover, the author shows that the presence of the singularity on the boundary of the domain allows to slightly enlarge the critical value for the constant $\mu$, up to $\mu^*:=N^2/4$. 
\\
\\
\indent In this article we analyse the case of a potential with singularity distributed all over the boundary. To the best of our knowledge, this is a problem that has never been treated in precedence, although it is a natural generalisation of the results of the works presented above.
\\
\\
\indent This paper is organized as follows: in Section \ref{h-p_ineq} we present the classical Hardy-Poincar\'e inequality introduced by Brezis and Marcus in \citep{brezis1997hardy}, which will then be applied for obtaining well-posedness of the equation we consider; we also give some extensions of this inequality, needed for obtaining the Carleman estimate. These results are then employed for obtaining the well-posedness of our equation, applying classical semi-group theory. In Section \ref{carleman sec} we present the Carleman estimate, showing what are the main differences between our result and previous ones obtained, for instance, in \cite{ervedoza2008control}, \cite{vancostenoble2008null} and, later, in \cite{cazacu2014controllability}. In Section \ref{obs_ineq_proof} we derive the observability inequality \eqref{observ ineq} and we apply it in the proof of Theorem \ref{control thm}. In Section \ref{stabilisation} we prove that the bound $1/4$ for the Hardy constant $\mu$ is sharp for control, showing the impossibility of preventing the solutions of the equation from blowing-up in the case of supercritical potentials. The Carleman estimates is proved in Section \ref{carleman_proof}. Section \ref{open pb} is dedicated to some interesting open problems related to our results. Finally, we conclude our article with an appendix  in which we prove several technical Lemmas that are fundamental in our analysis. 
%section 2
\section{Hardy-Poincaré inequalities and well-posedness}\label{h-p_ineq}
When dealing with equations involving singular inverse-square potentials, it is by now classical that of great importance is an Hardy-type inequality. Inequalities of this kind have been proved to hold also in the more general case of for the potential $\mu/\delta^2$ (see, for instance \cite{brezis1997hardy},\cite{marcus1998best}); in particular, we have
\begin{proposition}\label{hardy_dist prop}
Let $\Omega\subset\RR^N$ be a bounded $C^2$ domain; then, for any $u\in H_0^1(\Omega)$, and for any $\mu\leq 1/4$, the following inequality holds
\begin{align}\label{hardy_dist}
	\int_{\Omega} |\nabla u|^2\,dx \geq \mu\int_{\Omega} \frac{u^2}{\delta^2}\,dx.
\end{align} 
\end{proposition}
Inequality \eqref{hardy_dist} will be applied for obtaining the well-posedness of \eqref{heat hardy nh}, as well as the observability inequality \eqref{observ ineq}. For obtaining the Carleman estimate, instead, we are going to need the following Propositions
\begin{proposition}\label{hardy 1 prop}
Let $\Omega\subset\RR^N$ be a bounded $C^2$ domain. For any $\mu\leq 1/4$ and any $\gamma\in (0,2)$ there exist two positive constants $A_1$ and $A_2$, depending on $\gamma$ and $\Omega$ such that, for any $u\in H_0^1(\Omega)$, the following inequality holds
\begin{align}\label{hardy 1}
	A_1\int_{\Omega}\frac{u^2}{\delta^{\gamma}}\,dx + \mu\int_{\Omega}\frac{u^2}{\delta^2}\,dx \leq \int_{\Omega}|\nabla u|^2\,dx + A_2\int_{\Omega} u^2\,dx.
\end{align}
\end{proposition}

\begin{proposition}\label{hardy 2 prop}
Let $\Omega\subset\RR^N$ be a bounded $C^2$ domain. For any $\mu\leq 1/4$ and any $\gamma\in (0,2)$ there exists a positive constant $A_3$ depending on $\gamma$, $\mu$ and $\Omega$ such that, for any $u\in H_0^1(\Omega)$, the following inequality holds
\begin{align}\label{hardy 2}
	\int_{\Omega} \delta^{2-\gamma}|\nabla u|^2\,dx \leq R_{\Omega}^{2-\gamma} \int_{\Omega}\left(|\nabla u|^2 - \mu\frac{u^2}{\delta^2}\right)\,dx + A_3\int_{\Omega} u^2\,dx.
\end{align}
\end{proposition}

\begin{proposition}\label{hardy 3 prop}
Let $\Omega\subset\RR^N$ be a bounded $C^2$ domain. For any $\mu\leq 1/4$ and any $\gamma\in (0,2)$ there exist two positive constants $A_4$ and $A_5$ depending on $\gamma$, $\mu$ and $\Omega$ such that, for any $u\in H_0^1(\Omega)$, the following inequality holds
\begin{align}\label{hardy 3}
	\int_{\Omega}\left(|\nabla u|^2 -\mu\frac{u^2}{\delta^2}\right)\,dx + A_4\int_{\Omega} u^2\,dx \geq A_5\int_{\Omega}\left(\delta^{2-\gamma}|\nabla u|^2 + A_1\frac{u^2}{\delta^{\gamma}}\right)\,dx,
\end{align}
where $A_1$ is the positive constant introduced in Proposition \ref{hardy 1 prop}.
\end{proposition}
The proof of \ref{hardy 1} follows immediately from the inequalities with weighted integral presented in \cite[Section 4]{brezis1997hardy} and we are going to omit it here; moreover, \ref{hardy 3} is a direct consequence of the application of \ref{hardy 1} and \ref{hardy 2}. Concerning the proof of Proposition \ref{hardy 2 prop}, instead, we will presented it in appendix B.
\\
\\
\indent We conclude this section analysing existence and uniqueness of solutions for equation \eqref{heat hardy nh}, applying classical semi-group theory; at this purpose, we apply the same argument presented in \cite{cazacu2014controllability}. Therefore, for any fixed $\gamma\in[0,2)$ let us define the set 
\begin{align}\label{set_gamma}
	\mathcal{L}^{\gamma}:=\left\{A>0 \; \textrm{ s.t. } \; \inf_{u\in H_0^1(\Omega)} \frac{\int_{\Omega}\left(|\nabla u|^2 - \mu^* u^2/\delta^2 + Au^2\right)\,dx}{A_1\int_{\Omega}u^2/\delta^{\gamma}\,dx}\geq 1\right\}.
\end{align}
\indent We remind here that $\mu^*$ is the critical Hardy constant and that in our case we have $\mu^*=1/4$. Moreover, the set \eqref{set_gamma} is clearly non empty since it contains the constant $A_2$ in the inequality \eqref{hardy 1}. Now, we define 
\begin{align}\label{A_zero_gamma}
	A_0^{\gamma}:=\inf_{A\in \mathcal{L}^{\gamma}} A
\end{align}
and, for any $\mu\leq\mu^*$, we introduce the functional 
\begin{align*}
	\Phi_{\mu}^{\gamma}(u):= \int_{\Omega} |\nabla u|^2\,dx - \mu\int_{\Omega}\frac{u^2}{\delta^2}\, dx + A_0^{\gamma} \int_{\Omega} u^2\,dx; 
\end{align*}
we remark that this functional is positive for any test function, due to \eqref{hardy 1} and to the particular choice of the constant $A_0^{\gamma}$. 
\\
\indent Next, let us define the Hilbert space $H_{\mu}^{\gamma}$ as the closure of $C_0^{\infty}(\Omega)$ with respect to the norm induced by $\Phi_{\mu}^{\gamma}$; if $\mu\leq \mu^*$ we obtain
\begin{align}\label{norm_equiv}
	\left(1-\frac{\mu^+}{\mu^*}\right)\int_{\Omega} \left(|\nabla u|^2 + A_0^{\gamma} u^2\right)\,dx + \frac{\mu^+}{\mu^*}\int_{\Omega} \frac{u^2}{\delta^{\gamma}}\,dx \leq \norm{u}{H}^2 \leq \left(1+\frac{\mu^-}{\mu^*}\right)\int_{\Omega} \left(|\nabla u|^2 + A_0^{\gamma} u^2\right)\,dx,
\end{align}
where $\mu^+:=\max\{0,\mu\}$ and $\mu^-:=\max\{0,-\mu\}$. 
\\
\indent  From the norm equivalence \eqref{norm_equiv}, in the sub-critical case $\mu<\mu^*$ it follows the identification $H_{\mu}^{\gamma}=H_0^1(\Omega)$; in the critical case $\mu=\mu^*$, instead, this identification does not hold anymore and the space $H_{\mu}^{\gamma}$ is slightly larger than $H_0^1(\Omega).$ For more details on the characterisation of these kind of spaces, we refer to \cite{vazquez2000hardy}.
\\
Let us now consider the unbounded operator $\mathcal{B}_{\mu}^{\gamma}:\mathcal{D}(\mathcal{B}_{\mu}^{\gamma})\subset L^2(\Omega)\to L^2(\Omega)$ defined as
\begin{align}\label{operator B}
	\begin{array}{c}
		\displaystyle\mathcal{D}(\mathcal{B}_{\mu}^{\gamma}):=\left\{u\in H_{\mu}^{\gamma}\;\textrm{ s.t. }\; -\Delta u-\frac{\mu}{\delta^2}u+A_0^{\gamma} u\in L^2(\Omega)\right\}, 
		\\
		\\
		\displaystyle\mathcal{B}_{\mu}^{\gamma}u:=-\Delta u-\frac{\mu}{\delta^2}u+A_0^{\gamma} u,
	\end{array}
\end{align} 
whose norm is given by
\begin{align*}
	\norm{u}{\mathcal{B}_{\mu}^{\gamma}} = \norm{u}{L^2(\Omega)} + \norm{\mathcal{B}_{\mu}^{\gamma} u}{L^2(\Omega)}.
\end{align*}
$\indent$ With the definitions we just gave, by standard semi-group theory we have that for any $\mu\leq\mu^*$ the operator $(\mathcal{B}_{\mu}^{\gamma},\mathcal{D}(\mathcal{B}_{\mu}^{\gamma}))$ generates an analytic semi-group in the pivot space $L^2(\Omega)$ for the equation \eqref{heat hardy nh}. For more details we refer to the Hille-Yosida theory, presented in \cite[Chapter 7]{brezis2010functional}, which can be adapted in the context of the space $H_{\mu}^{\gamma}$ introduced above. 
\\
\indent Therefore, from the construction we just presented we immediately have the following well-posedness result 
\begin{theorem}\label{well-posedness thm }
Given $u_0\in L^2(\Omega)$ and $f\in C([0,T];L^2(\Omega))$, for any $\mu\leq 1/4$ the problem \eqref{heat hardy nh} admits a unique weak solution 
	\begin{align*}
		u\in\mathrm{C}^0([0,T];L^2(\Omega))\cap\mathrm{L}^2((0,T);H_{\mu}^{\gamma}).
	\end{align*}
\end{theorem}
%section 3
\section{Carleman estimate}\label{carleman sec}
\subsection{Choice of the weight $\sigma$}
The observability inequality \eqref{observ ineq} will be proved, as it is classical in controllability problems for parabolic equations, applying a Carleman estimate. 
\\
\\
\indent The main problem when designing a Carleman estimate is the choice of a proper weight function $\sigma(x,t)$. In our case, this $\sigma$ will be an adaptation of the one used in \citep{cazacu2014controllability}, that we conveniently modify in order to deal with the presence of the singularities distributed all over the boundary. In particular, the weight we propose is the following 
\begin{align}\label{sigma}
	\sigma(x,t)=\theta(t)\left(C_{\lambda}-\delta^2\psi-\left(\frac{\delta}{r_0}\right)^{\lambda}\phi\right), \;\;\; \phi=e^{\lambda\psi},
\end{align}
where 
\begin{align}\label{theta}
	\theta(t)=\left(\frac{1}{t(T-t)}\right)^3.
\end{align}
\indent Here, $C_{\lambda}$ is a positive constant large enough as to ensure the positivity of $\sigma$, and $\lambda$ is a positive parameter aimed to be large; besides, $r_0$ satisfies 
\begin{align}\label{r_0}
	\nonumber \displaystyle r_0 \leq\min & \left\{1, \frac{2|\psi|_{\infty}}{4|D\psi|_{\infty}+|D^2\psi|_{\infty}},\, \frac{1}{R_{\Omega}\sqrt{4|D\psi|_{\infty}^2+2|D^2\psi|_{\infty}}},\, \frac{|\psi|_{\infty}}{2(2-\gamma)|D\psi|_{\infty}},\, \left(\frac{M_2}{4|\mu||D\psi|_{\infty}}\right)^{1/(\gamma-1)}, \right.
	\\
	\nonumber &\displaystyle \left. \frac{1}{\sqrt{8\Do{\psi_1}|D\psi|_{\infty}/\varpi_0+3|D^2\psi|_{\infty}}},\, \frac{2|\psi|_{\infty}}{|D\psi|_{\infty}^2+(1+2|\psi|_{\infty})|D\psi|_{\infty}},\, \frac{1}{|D\psi|_{\infty}^2+2|D\psi|_{\infty}},  \right.
	\\
	&\displaystyle \left. \frac{3|\psi|_{\infty}^2}{4|D\psi|_{\infty}},\, \frac{1}{|D\psi|_{\infty}\sqrt{D_3|\psi|_{\infty}^2+D_4}}\right\},
\end{align}
where $\gamma$ is the parameter appearing in the Hardy inequalities presented above, with the particular choice $\gamma\in(1,2)$, while $M_2$ is a positive constant that will be introduced later. The choice of $r_0$ as in \eqref{r_0} is motivated by technical reasons that will be carefully justified throughout the paper.  
Finally, $\psi$ is a bounded regular function (at least $C^4(\overline{\Omega})$) defined as
\begin{align}\label{psi}
	\psi=\varpi(\psi_1+1),
\end{align}
with $\psi_1\in C^4(\overline{\Omega})$ and bounded, satisfying the conditions
\begin{align}\label{psi_1}
	\left\{\begin{array}{ll}
		\psi_1(x)=\delta(x) & \forall x\in\Omega_{r_0},
		\\ \psi_1(x)>r_0 & \forall x\in\Omega\setminus\overline{\Omega_{r_0}},
		\\ \psi_1(x)=r_0 & \forall x\in\Sigma_{r_0},
		\\ |\nabla\psi_1(x)|\geq\varpi_0>0 & \forall x\in\Omega\setminus\overline{\omega_0},
	\end{array}\right.
\end{align}
for $\varpi\varpi_0>2C_{\Omega}$, where $C_{\Omega}$ is the constant introduced in \cite[Section 2]{cazacu2014controllability}. Such function exists but its construction is not trivial. See \cite[Section 2]{cazacu2014controllability} for more details. In particular, under these conditions $\psi$ satisfies the following useful properties
\begin{align}\label{psi_prop}
	\left\{\begin{array}{ll}
		\psi(x)=1 & \forall x\in\Gamma,
		\\ \psi(x)>1 & \forall x\in\Omega,
		\\ |\nabla\psi(x)|\geq 2C_{\Omega} & \forall x\in\Omega\setminus\overline{\omega_0},
	\end{array}\right.
\end{align}
\indent In \eqref{psi_1} and \eqref{psi_prop}, $\omega_0\subset\subset\omega$ is a non-empty subset of the control region $\omega$; moreover, due to technical computations, we fix $\varpi$ such that
\begin{align}\label{varpi}
	\nonumber\varpi\geq\max\left\{1, \frac{1}{\varpi_0^2}\left(1+\frac{2\Do{\psi_1}}{r_0}+|D^2\psi|_{\infty}\right), \frac{2}{\varpi_0^2}\left(1+\frac{2\Do{\psi_1}}{r_0}\right), \frac{4\Do{\psi_1}}{\varpi_0^2}, \frac{24\Do{\psi_1}R_{\Omega}}{\varpi_0^2}, \frac{2}{\varpi_0}\right\},
	\\
\end{align}
where $R_{\Omega}$ is the diameter of the domain $\Omega$, while $\Do{\psi_1}$ is a positive constant that will be introduced later. Furthermore, throughout the paper, formally, for a given function $f$ we apply the notations
\begin{align}\label{notation f}
	\nonumber & |f|_{\infty}:=\norm{f}{L^{\infty}(\Omega)}, && |Df|_{\infty}:=\norm{\nabla f}{L^{\infty}(\Omega)},
	\\
	&D^2f(\xi,\xi):=\sum_{i,j=1}^N \partial^2_{x_ix_j}f\xi_i\xi_j, \;\;\; \forall\xi\in\RR^N, && |D^2f|_{\infty}:=\sum_{i,j=1}^N\norm{\partial^2_{x_ix_j}f}{L^{\infty}(\Omega)},
\end{align}
and we denote 
\begin{align}\label{notation set}
	\Omega_{r_0}:=\{\left.x\in\Omega\,\right|\delta(x)<r_0\}, \;\;\; \mathcal{O}:=\Omega\setminus\left(\overline{\omega_0}\cup\overline{\Omega_{r_0}}\,\right), \;\;\; \tilde{\mathcal{O}}:=\Omega\setminus\overline{\Omega_{r_0}}.
\end{align}

\subsection{Motivation for the choice of $\sigma$}

The weigh $\sigma$ that we propose for our Carleman estimates is not the standard one; we had to modify it in order to deal with some critical terms that emerge in our computations due to the presence of the singular potential. We justify here our choice, highlighting the reasons why the weights presented in previous works (\cite{cazacu2014controllability},\cite{ervedoza2008control},\cite{fursikov1996controllability}) are not suitable for the problem we consider.
\\
\indent In general, the weight used to obtain Carleman estimates for parabolic equations is assumed to be positive and to blow-up at the extrema of the time interval; besides, it has to be taken in separated variables. Therefore, we are looking for a function $\sigma(x,t)$ satisfying 
\begin{subequations}
\label{eq:tot}
\begin{empheq}[left={}\empheqlbrace]{align}
		&\sigma(x,t)=\theta(t)p(x), 	& & (x,t)\in Q, \\
		&\sigma(x,t)>0, & & (x,t)\in Q,  \label{sigma prop 1}\\
		&\displaystyle\lim_{t\to 0^+}\sigma(x,t) = \lim_{t\to T^-}\sigma(x,t) = +\infty, & & x\in\Omega. \label{sigma prop 2}
\end{empheq}
\end{subequations}
\\
The function $\theta$ is usually chosen in the form 
\begin{align*}
	\theta(t)=\left(\frac{1}{t(T-t)}\right)^k
\end{align*}
for $k\geq 1$, and this choice in particular ensures the validity of \eqref{sigma prop 2}; in our case we assume $k=3$ which, as we will remark later, is the minimum value for obtaining some important estimates that we need in the proof of the Carleman inequality. 
\\
\indent While the choice of $\theta$ is standard, the main difficulty when building a proper $\sigma$ is to identify a suitable $p(x)$ which is able to deal with the specificity of the equation we are analysing. 
\\
\indent In \cite{fursikov1996controllability}, Fursikov and Imanuvilov obtained the controllability of the standard heat equation employing a positive weight in the form 
\begin{align*}
	\sigma_1=\theta(t)\left(C_{\lambda}-e^{\lambda\psi}\right),
\end{align*}
with a function $\psi\in C^2(\overline{\Omega})$ satisfying
\begin{align*}
	\left\{\begin{array}{ll}
		\psi(x)>0, & x\in\Omega,
		\\ \psi(x)=0, & x\in\partial\Omega,
		\\ |\nabla\psi(x)|>0, & x\in\overline{\Omega}\setminus\omega_0.
	\end{array}\right.
\end{align*} 
\indent An example of a $\psi$ with this behaviour is shown in Figure 1 below; in particular, we notice that this function is required to be always strictly monotone outside of the control region. 
\begin{figure}[!h]
	\centering	
	\includegraphics[scale=0.3]{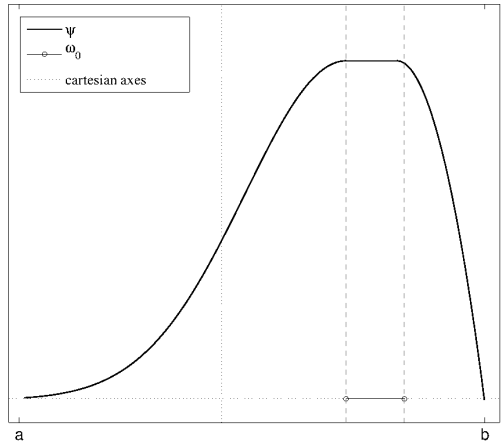}
	\caption{Function $\psi$ of Fursikov and Imanuvilov in one space dimension on the interval $(a,b)$}
\end{figure}
\\
\indent This standard weight was later modified by Ervedoza in \cite{ervedoza2008control}, for dealing with problems with interior quadratic singularities; in this case, the author applies the weight
\begin{align*}
	\sigma_2=\theta(t)\left(C_{\lambda}-\frac{1}{2}|x|^2-e^{\lambda\psi(x)}\right),
\end{align*}
with a function $\psi$ such that 
\begin{align*}
	\left\{\begin{array}{ll}
		\psi(x)=\ln(|x|), & x\in B(0,1),
		\\ \psi(x)=0, & x\in\partial\Omega,
		\\ \psi(x)>0, & x\in\Omega\setminus\overline{B}(0,1),
		\\ |\nabla\psi(x)|\geq\gamma>0, & x\in\overline{\Omega}\setminus\omega_0.
	\end{array}\right.
\end{align*} 
\begin{figure}[!h]
	\centering	
	\includegraphics[scale=0.33]{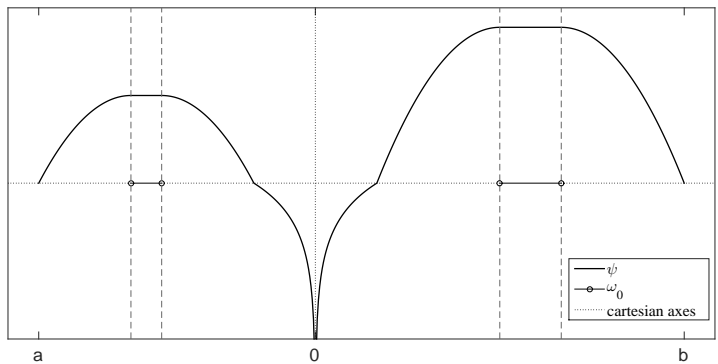}
	\caption{Function $\psi$ of Ervedoza in one space dimension on the interval $(a,b)$}
\end{figure}
\\
\indent This choice is motivated by some critical terms appearing due to the presence of the potential, that must be absorbed outside $\omega$ in the Carleman estimate (see \cite[Eq. 2.15]{ervedoza2008control}). 
\\
\indent In particular, in order to take advantage of the Hardy inequality, the author needs to get rid of singular terms in the form $\Delta\sigma/|x|^2$ and $(x\cdot\nabla\sigma)/|x|^4$. The weight proposed allows to deal with this terms; indeed near the singularity, when $\lambda$ is large enough $\sigma_2$ behaves like 
\begin{align*}
	\sigma_2\sim\theta(t)\left(C_{\lambda}-\frac{1}{2}|x|^2\right),
\end{align*}
which is the weight employed by Vancostenoble and Zuazua in \cite{vancostenoble2009hardy} for their proof of the controllability of the heat equation with a singular potential and which satisfies
$\nabla\sigma_2\sim x$ and $\Delta\sigma_2\sim C$ as $x\to 0$. On the other hand, away from the origin, where no correction is needed, $\sigma_2$ maintains the behaviour of the classical weight $\sigma_1$. 
\\
\indent A further modification is proposed by Cazacu in \cite{cazacu2014controllability}, in the case of an equation with boundary singularity. In this case, indeed, the terms $\Delta\sigma/|x|^2$ and $(x\cdot\nabla\sigma)/|x|^4$ generates singularities that cannot be absorbed in a neighbourhood of the origin employing $\sigma_2$, since this weight involves a function $\psi$ which is assumed to be zero on the boundary. Therefore, the author proposes a new weight 
\begin{align*}
	\sigma_3=\theta(t)\left(C_{\lambda}-|x|^2\psi-\left(\frac{|x|}{r_0}\right)^{\lambda}e^{\lambda\psi}\right),
\end{align*}
where $\psi$ is now chosen as in \eqref{psi}, with the fundamental property of being constant and positive on the boundary. 
\begin{figure}[!h]
	\centering
	\includegraphics[scale=0.33]{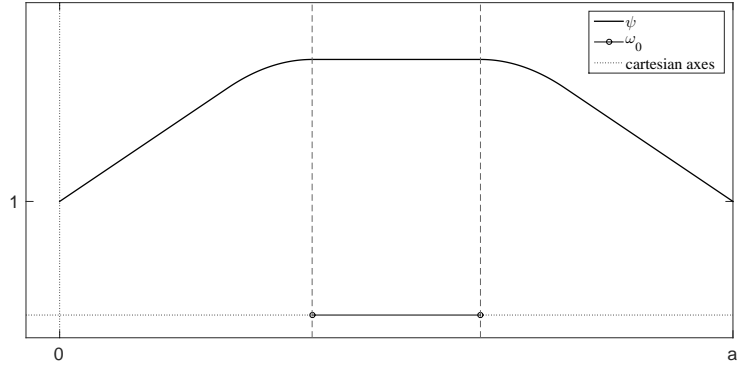}
	\caption{Function $\psi$ of Cazacu in one space dimension on the interval $(0,a)$}
\end{figure}
\\
\indent Finally, when dealing as in our case with a singularity distributed all over the boundary the weights presented above do not allow anymore to manage properly the terms containing the singularities, since they now have a different nature. Therefore, we need to introduce further modifications in the weight we want to employ, designing it in a way that could compensate this kind of degeneracies. At this purpose, it is sufficient to modify $\sigma_3$ replacing the terms of the form $|x|$ with the distance function $\delta$; being still in the case of boundary singularities the function $\psi$ introduced in \cite{cazacu2014controllability} (see \eqref{psi} above) turns out to be a suitable one also in our case.
\\
\indent For concluding, we want to emphasise the fact that all the changes in the classical weight we introduced above are purely local, around the points where the singularity of the potential arises. This, of course, because as long as the potential remains bounded it can be handled with the same techniques as for the classical heat equation.
\\
\\
We now have all we need for introducing the Carleman estimate. 
\begin{theorem}\label{carleman thm}
Let $\sigma$ be the weight defined in \eqref{sigma}. There exist two positive constants $\lambda_0$ and $\mathcal{M}$ such that for any $\lambda\geq\lambda_0$ there exists $R_0=R_0(\lambda)$ such that for any $R\geq R_0$ and for any solution $v$ of \eqref{heat hardy adj} it holds 
\begin{align}\label{carleman}
	\nonumber R\int_Q &\theta e^{-2R\sigma}\left(\delta^{2-\gamma}|\nabla v|^2+A_1\frac{v^2}{\delta^{\gamma}}\right)\,dxdt + \lambda R\intr{\Omega_{r_0}\times(0,T)} \theta\left(\frac{\delta}{r_0}\right)^{\lambda-2}e^{-2R\sigma}|\nabla v|^2\,dxdt  
	\\
	\nonumber + \lambda^2 & R\intr{\mathcal{O}\times(0,T)} \theta\left(\frac{\delta}{r_0}\right)^{\lambda}\phi e^{-2R\sigma}|\nabla v|^2\,dxdt + R^3\intr{\Omega_{r_0}\times(0,T)}\theta^3\delta^2e^{-2R\sigma}v^2\,dxdt 
	\\
	\nonumber + \lambda^4 & R^3\intr{\mathcal{O}\times(0,T)} \theta^3\left(\frac{\delta}{r_0}\right)^{3\lambda}\phi^3e^{-2R\sigma}v^2\,dxdt 
	\\
	\nonumber &\leq \mathcal{M}\left(\lambda^4R^3\intr{\omega_0\times(0,T)}  \theta^3\left(\frac{\delta}{r_0}\right)^{3\lambda}\phi^3e^{-2R\sigma}v^2\,dxdt + \lambda^2R\intr{\omega_0\times(0,T)} \theta\left(\frac{\delta}{r_0}\right)^{\lambda}\phi e^{-2R\sigma}|\nabla v|^2\,dxdt\right)
	\\
\end{align}
\end{theorem}
The proof of Theorem \ref{carleman thm} is very technical and will be presented in Section \ref{carleman_proof}. It relies on several technical Lemmas that we are going to prove in the appendix.  
%section 4
\section{Proof of the observability inequality \eqref{observ ineq} and of the controllability Theorem \ref{control thm}}\label{obs_ineq_proof}
We now apply the Carleman estimate we just obtained for proving the observability inequality \eqref{observ ineq}. This inequality will then be employed in the proof of our main result, Theorem \ref{control thm}. 
\begin{proof}[Prooof of the observability inequality \eqref{observ ineq}] 
Let us fix $\lambda\geq\lambda_0$ and $R\geq R_0(\lambda)$ such that \eqref{carleman} holds. These parameters now enter in the constant $\mathcal{M}$; in particular we have 
\begin{align*}
	\int_Q &\theta e^{-2R\sigma}\frac{v^2}{\delta^{\gamma}}\,dxdt \leq \mathcal{M}\left(\,\intr{\omega_0\times(0,T)}  \theta^3\phi^3e^{-2R\sigma}v^2\,dxdt + \intr{\omega_0\times(0,T)} \theta\phi e^{-2R\sigma}|\nabla v|^2\,dxdt\,\right).
\end{align*}
Now, it is straightforward to check that there exists a positive constant $\mathcal{P}$ such that
\begin{align*}
	\begin{array}{ll}
		\theta e^{-2R\sigma}\frac{1}{\delta^{\gamma}}\geq\mathcal{P}, & (x,t)\in\Omega\times\left[\frac{T}{4},\frac{3T}{4}\right],
		\\
		\\
		\theta^3\phi^3 e^{-2R\sigma}\leq\mathcal{P}, & (x,t)\in\omega_0\times(0,T),
		\\
		\\
		\theta\phi e^{-2R\sigma}\leq\mathcal{P}e^{-R\sigma}, & (x,t)\in\omega_0\times(0,T).
	\end{array}
\end{align*}
Thus the inequality above becomes
\begin{align*}
	\int_{\frac{T}{4}}^{\frac{3T}{4}}\int_{\Omega} v^2\,dxdt\leq\mathcal{N} \left(\,\intr{\omega_0\times(0,T)}  v^2\,dxdt + \intr{\omega_0\times(0,T)} e^{-R\sigma}|\nabla v|^2\,dxdt\,\right).
\end{align*}
Moreover, multiplying equation \eqref{heat hardy adj} by $v$ and integrating over $\Omega$ we obtain
\begin{align*}
	\frac{1}{2}\frac{d}{dt}\int_{\Omega} v^2\,dx = \int_{\Omega} |\nabla v|^2\,dx - \mu\int_{\Omega}\frac{v^2}{\delta^2}\,dx,
\end{align*}
which, applying \eqref{hardy_dist}, implies
\begin{align*}
	\frac{d}{dt}\int_{\Omega} v^2\,dx \geq -C\int_{\Omega} v^2\,dx.
\end{align*}
Hence, the function $t\mapsto e^{2Ct}\norm{v(\cdot,t)}{L^2(\Omega)}$ is increasing, that is
\begin{align*}
	e^{-2CT}\int_{\Omega} v(x,0)^2\,dx\leq \int_{\Omega} v(x,t)^2\,dx,
\end{align*}
and, integrating in time between $T/4$ and $3T/4$ we have
\begin{align*}
	\frac{T}{2}e^{-2CT}\int_{\Omega} v(x,0)^2\,dx\leq \int_{\frac{T}{4}}^{\frac{3T}{4}}\int_{\Omega} v(x,t)^2\,dx.
\end{align*}
Thus, we obtain the inequality
\begin{align*}
	\int_{\Omega} v(x,0)^2\,dxdt\leq\frac{2\mathcal{N}e^{2CT}}{T} \left(\,\intr{\omega_0\times(0,T)}  v^2\,dxdt + \intr{\omega_0\times(0,T)} e^{-R\sigma}|\nabla v|^2\,dxdt\,\right).
\end{align*}
Therefore to conclude the proof of \eqref{observ ineq}, it is sufficient to apply the following lemma:
\begin{lemma}[Cacciopoli’s inequality]\label{caccioppoli lemma}
Let $\bar{\sigma}: (0,T)\times\omega_0\to\RR_+^*$ be a smooth non-negative function such that
\begin{align*}
	\bar{\sigma}(x,t)\to +\infty, \;\; \textrm{ as } t\to 0^+ \;\; \textrm{ and as } t\to T^-,
\end{align*}
and let $\mu\leq\mu^*$. Then, there exists a constant $\Upsilon$ independent of $\mu$ such that any solution $v$ of \eqref{heat hardy adj} satisfies
\begin{align}\label{caccioppoli}
	\intr{\omega_0\times(0,T)} e^{-R\bar{\sigma}}|\nabla v|^2\,dxdt \leq \Upsilon\intr{\omega\times(0,T)}  v^2\,dxdt. 
\end{align}
\end{lemma}
Lemma \ref{caccioppoli lemma} is a trivial adaptation of an analogous result, \cite[Lemma 3.3]{vancostenoble2008null}, and its proof is left to the reader. It is now straightforward that, applying \eqref{caccioppoli} for $\sigma$ as in \eqref{sigma} we finally get
\begin{align*}
	\int_{\Omega} v(x,0)^2\,dxdt\leq C_T\intr{\omega_0\times(0,T)}  v^2\,dxdt,
\end{align*}
that clearly implies \eqref{observ ineq}, due to the definition of $\omega_0$.
\end{proof}

\begin{proof}[Proof of Theorem \eqref{control thm}]
Once the observability inequality \eqref{observ ineq} is known to hold, we can immediately obtain the controllability of our equation through a control $f\in L^2(\omega\times(0,T))$. To do that, we are going to introduce the functional
\begin{align}\label{control functional}
	J(v_T):=\frac{1}{2}\intr{\omega\times (0,T)} v^2\,dxdt + \int_{\Omega} v(x,0)u_0(x)\,dx,
\end{align}
defined over the Hilbert space 
\begin{align}
	H:=\left\{\left.v_T\in L^2(\Omega)\,\right| \textrm{ the solution } v \textrm{ of } \eqref{heat hardy adj} \textrm{ satisfies } \intr{\omega\times(0,T)} v^2\,dxdt\leq +\infty \right\}.
\end{align}
To be more precise, $H$ is the completion of $L^2(\Omega)$ with respect to the norm $\,\left(\int_0^T\int_{\omega} v^2\,dxdt\right)^{1/2}.$
\indent Observe that $J$ is convex and, according to \eqref{observ ineq}, it is also continuous in $H$; on the other hand, again \eqref{observ ineq} gives us also the coercivity of $J$. Therefore, there exists $v^*\in H$ minimizing $J$. 
\\
The corresponding Euler-Lagrange equation is 
\begin{align}\label{euler-lagrange}
	\intr{\omega\times (0,T)} v(x,t)F(x,t)\,dxdt + \int_{\Omega}u_0(x)v(x,0)\,dx = 0,
\end{align}
where $F(x,t):=v^*(x,t)\chi_{\omega}$. $F$ will be our control function; we observe that, by definition $F\in L^2(\omega\times (0,T))$. Now, considering equation \eqref{heat hardy nh} with $f=F$, multiplying it by $v$ and integrating by parts, we get 
\begin{align*}
	\int_0^1 u(x,T)v_T(x)\,dx = \intr{\omega\times (0,T)} v(x,t)F(x,t)\,dxdt + \int_{\Omega}u_0(x)v(x,0)\,dx,
\end{align*}
for any $v_T\in L^2(\Omega)$. Hence, from \eqref{euler-lagrange} we immediately conclude $u(x,T)=0$.
\end{proof}
%section 5
\section{Non existence of a control in the supercritical case}\label{stabilisation}
\indent As we mentioned before, in \cite{cabre1999existence} is proved that in the super-critical case, i.e. for $\mu>1/4$, the Cauchy problem for our singular heat equation is severely ill-posed. However, a priori this fact does not exclude that, given $u_0\in L^2(\Omega)$, it is possible to find a control $f\in L^2((0,T);L^2(\Omega))$ localised in $\omega$ such that there exists a solution of \eqref{heat hardy nh}. If this fact occurs, it would mean that we can prevent blow-up phenomena by acting on a subset of the domain. 
\\
\indent However, as we are going to show in this section, this control function $f$ turns out to be impossible to find for $\mu>1/4$ and, in this case, we cannot prevent the system from blowing up. Therefore, the upper bound $1/4$ for the Hardy constant $\mu$ shows up to be sharp for control. 
\\
\indent The proof of this fact will rely on an analogous result presented in \cite{ervedoza2008control}; therefore, following the ideas of optimal control, for any $u_0\in L^2(\Omega)$ we consider the functional
\begin{align*}
	J_{u_0}(u,f\,):=\frac{1}{2}\int_Q |u(x,t)|^2\,dxdt+\frac{1}{2}\int_0^T \norm{f(t)}{L^2(\Omega)}^2\,dt,
\end{align*}
defined on the set
\begin{align*}
	\mathcal{C}(u_0):=\left\{ \left.(u,f\,)\in L^2((0,T),H_0^1(\Omega))\times L^2((0,T),L^2(\Omega))\;\right| u \textrm{ satisfies } \eqref{heat hardy nh}\right\}.
\end{align*}
We say that it is possible to stabilise system \eqref{heat hardy nh} if we can find a constant $A$ such that
\begin{align*}
	\inf_{(u,f\,)\,\in\,\mathcal{C}(u_0)} J_{u_0}(u,f\,)\leq A\norm{u_0}{L^2(\Omega)}^2.
\end{align*}
Now, for $\varepsilon>0$, we approximate \eqref{heat hardy nh} by the system 
\begin{align}\label{heat hardy nh eps}
	\renewcommand*{\arraystretch}{1.3}
	\left\{\begin{array}{ll}
		\displaystyle u_t-\Delta u-\frac{\mu}{\delta^2+\varepsilon^2}u=f, & (x,t)\in Q
		\\ u=0, & (x,t)\in\Gamma\times(0,T)
		\\ u(x,0)=u_0(x), & x\in\Omega,
	\end{array}\right.
\end{align}
\indent Due to the boundedness of the potential, \eqref{heat hardy nh eps} is well-posed; therefore, we can define the functional
\begin{align*}
	J_{u_0}^{\varepsilon}(f):=\frac{1}{2}\int_Q |u(x,t)|^2\,dxdt+\frac{1}{2}\int_0^T \norm{f(t)}{L^2(\Omega)}^2\,dt,
\end{align*}
where $f\in L^2((0,T);L^2(\Omega))$ is localised in $\omega$ and $u$ is the corresponding solution of \eqref{heat hardy nh eps}. We are going to prove the following
\begin{theorem}\label{stablisation thm}
Assume that $\mu>1/4$. There is no constant $A$ such that, for all $\varepsilon>0$ and all $u_0\in L^2(\Omega)$,
\begin{align*}
	\inf_{f\in L^2((0,T);L^2(\Omega))} J_{u_0}^{\varepsilon}(f) \leq A\norm{u_0}{L^2(\Omega)}^2.
\end{align*}
\end{theorem}
We are going to prove Theorem \ref{stablisation thm} in two steps: firstly, we give some basic estimates on the spectrum of the operator
\begin{align}\label{operator L}
	\mathcal{L}^{\varepsilon}:= -\Delta - \frac{\mu}{\delta^2+\varepsilon^2}\,\mathcal{I}
\end{align}
on $\Omega$ with Dirichlet boundary conditions; secondly, we will apply these estimates for proving the main result of this section, Theorem \ref{stablisation thm}.

\subsection{Spectral estimates}
Since the function $1/(\delta^2+\varepsilon^2)$ is smooth and bounded in $\Omega$ for any $\varepsilon>0$, the spectrum of $\mathcal{L}^{\varepsilon}$ is given by a sequence of real eigenvalues $\lambda_0^{\varepsilon}\leq\lambda_1^{\varepsilon}\leq\ldots\leq\lambda_k^{\varepsilon}\leq\ldots\;$, with $\lambda_k^{\varepsilon}\to +\infty$ as $k\to+\infty$, to which corresponds a family of eigenfunctions $\phi_k^{\varepsilon}$ that forms an orthonormal basis of $L^2(\Omega)$. 
\begin{proposition}
Assume $\mu>1/4$ and let $\Omega_{\beta}$ be as in \eqref{notation set}. Then we have
\begin{align}\label{eigenv behaviour}
	\lim_{\varepsilon\to 0^+}\lambda_0^{\varepsilon}=-\infty
\end{align}
and, for all $\beta>0$,
\begin{align}\label{eigenf behaviour}
	\lim_{\varepsilon\to 0^+}\norm{\phi_0^{\varepsilon}}{H^1(\Omega\setminus\overline{\Omega}_{\beta})}=0.
\end{align}
\end{proposition}
\begin{proof}
We argue by contradiction and we assume that $\lambda_0^{\varepsilon}$ is bounded from below by some constant $M$. From the Rayleigh formula we have
\begin{align*}
	\mu\int_{\Omega}\frac{u^2}{\delta^2+\varepsilon^2}\,dx \leq\int_{\Omega}|\nabla u|^2\,dx - M\int_{\Omega}u^2\,dx,
\end{align*}
for all $\varepsilon>0$ and any $u\in H_0^1(\Omega)$. Taking now $u\in\mathcal{D}(\Omega)$, we pass to the limit as $\varepsilon\to 0^+$ in the inequality above and we get
\begin{align}\label{stabilisation est}
	\mu\int_{\Omega}\frac{u^2}{\delta^2}\,dx \leq\int_{\Omega}|\nabla u|^2\,dx - M\int_{\Omega}u^2\,dx,
\end{align}
that holds for any $u\in H_0^1(\Omega)$ by a density argument.
\\
\indent Now, given $\beta_0>0$, let us choose $u\in H_0^1(\Omega_{\beta_0})$, that we extend by zero on $\RR^N$, and let us define, for $a\geq 1$,
\begin{align*}
	u_a(x):=a^Nu(ax).
\end{align*}
\indent This function is clearly in $H_0^1(\Omega_{\beta_0})$, and consequently in $H_0^1(\Omega)$; therefore, we can apply \eqref{stabilisation est} to it and find
\begin{align*}
	a^2\left(\mu\int_{\Omega}\frac{u^2}{\delta^2}\,dx - \int_{\Omega}|\nabla u|^2\,dx \right)\leq - M\int_{\Omega}u^2\,dx.
\end{align*}
Passing to the limit as $a\to +\infty$, we obtain 
\begin{align*}
	\mu\int_{\Omega}\frac{u^2}{\delta^2}\,dx \leq\int_{\Omega}|\nabla u|^2\,dx,
\end{align*}
for any $u\in H_0^1(\Omega_{\beta_0})$. Therefore, we should have $\mu\leq 1/4$, since this is the Hardy inequality in the set $\Omega_{\beta_0}$; then, we have a contradiction.
\\
Now, consider the first eigenfunction $\phi_0^{\varepsilon}\in H_0^1(\Omega)$ of $\mathcal{L}^{\varepsilon}$, that by definition satisfies 
\begin{align}\label{eigen def}
	-\Delta\phi_0^{\varepsilon}-\mu\frac{\phi_0^{\varepsilon}}{\delta^2+\varepsilon^2}=\lambda_0^{\varepsilon}\phi_0^{\varepsilon},
\end{align}
in $\Omega$. Observe that, since the potential is smooth in $\Omega$, also the function $\phi_0^{\varepsilon}$ is smooth by classical elliptic regularity. 
\\
\indent Set $\beta>0$ and let $\xi_{\beta}$ be a non-negative smooth function, vanishing in $\Omega_{\beta/2}$ and equals to $1$ in $\RR^N\setminus\Omega_{\beta}$, with $\norm{\xi_{\beta}}{\infty}\leq 1$. Multiplying \ref{eigen def} by $\xi_{\beta}\phi_0^{\varepsilon}$ and integrating by parts we obtain
\begin{align}\label{cut_off eigen}
	\int_{\Omega} \xi_{\beta}\left|\nabla \phi_0^{\varepsilon}\right|^2\,dx + \left|\lambda_0^{\varepsilon}\right|\int_{\Omega} \xi_{\beta}\left(\phi_0^{\varepsilon}\right)^2\,dx = \mu\int_{\Omega} \xi_{\beta}\frac{\left(\phi_0^{\varepsilon}\right)^2}{\delta^2+\varepsilon^2}\,dx + \frac{1}{2}\int_{\Omega} \Delta\xi_{\beta}\left(\phi_0^{\varepsilon}\right)^2\,dx.
\end{align}
Therefore, since $\phi_0^{\varepsilon}$ is of unit $L^2$-norm, and due to the definition of $\xi_{\beta}$, we get
\begin{align*}
	\left|\lambda_0^{\varepsilon}\right|\intr{\Omega\setminus\Omega_{\beta}} \left(\phi_0^{\varepsilon}\right)^2\,dx \leq \frac{4\mu}{\beta^2} + \frac{1}{2}\,\norm{\Delta\xi_{\beta}}{L^{\infty}(\Omega)}.
\end{align*}
Since $\left|\lambda_0^{\varepsilon}\right|\to\infty$ as $\varepsilon\to 0^+$, we obtain that for any $\beta>0$
\begin{align}\label{eigenf int}
	\lim_{\varepsilon\to 0^+}\intr{\Omega\setminus\Omega_{\beta}} \left(\phi_0^{\varepsilon}\right)^2\,dx = 0. 
\end{align}
Furthermore, using again \eqref{cut_off eigen} and the definition of $\xi_{\beta}$
\begin{align*}
	\intr{\Omega\setminus\Omega_{\beta}} \left|\nabla \phi_0^{\varepsilon}\right|^2\,dx \leq \left(\frac{4\mu}{\beta^2} + \frac{1}{2}\,\norm{\Delta\xi_{\beta}}{L^{\infty}(\Omega)}\right) \intr{\Omega\setminus\Omega_{\beta/2}} \left(\phi_0^{\varepsilon}\right)^2\,dx.
\end{align*}
Hence, the proof of \eqref{eigenf behaviour} is completed by using \eqref{eigenf int} for $\beta/2$.
\end{proof}

\begin{proof}[Proof of Theorem \ref{stablisation thm}]
Fix $\varepsilon>0$ and choose $u_0^{\varepsilon}=\phi_0^{\varepsilon}$, that by definition is of unit $L^2$-norm. We want to show that 
\begin{align*}
	\inf_{f\in L^2((0,T);L^2(\Omega))} J_{u_0^{\varepsilon}}^{\varepsilon}(f) \to +\infty
\end{align*}
as $\varepsilon\to 0^+$.
\\
\indent Hence, let $f\in L^2((0,T);L^2(\Omega))$ and consider the corresponding solution $u$ of \eqref{heat hardy nh} with initial data $u_0^{\varepsilon}=\phi_0^{\varepsilon}$. Set
\begin{align*}
	\rho(t) = \int_{\Omega} u(x,t)\phi_0^{\varepsilon}(x)\,dx, \;\;\textrm{ and }\;\; \zeta(t)=\langle f(t),\phi_0^{\varepsilon}\rangle_{L^2(\Omega)};
\end{align*}
then, $\rho(t)$ satisfies the first order differential equation
\begin{align*}
	\left\{\begin{array}{l}
		\rho'(t)+\lambda_0^{\varepsilon}\,\rho(t) = \zeta(t),
		\\ \rho(0)=1.
	\end{array}\right.
\end{align*}
By the Duhamel's formula we obtain
\begin{align*}
	\rho(t) = e^{-\lambda_0^{\varepsilon}t} + \int_0^t e^{-\lambda_0^{\varepsilon}(t-s)}\zeta(s)\,ds.
\end{align*}
Therefore,
\begin{align}\label{ode est}
	\int_Q u^2\,dxdt \geq \int_0^T \rho(t)^2\,dt \geq \frac{1}{2}\int_0^T e^{-\lambda_0^{\varepsilon}t}\,dt - \int_0^T\left(\int_0^t e^{-\lambda_0^{\varepsilon}(t-s)}\zeta(s)\,ds\right)^2\,dt.
\end{align}
Of course
\begin{align*}
	\frac{1}{2}\int_0^T e^{-\lambda_0^{\varepsilon}t}\,dt = \frac{1}{4\lambda_0^{\varepsilon}}\left(e^{2\lambda_0^{\varepsilon}T}-1\right);
\end{align*}
on the other hand, by trivial computations we have 
\begin{align*}
	\int_0^T\left(\int_0^t e^{-\lambda_0^{\varepsilon}(t-s)}\zeta(s)\,ds\right)^2\,dt \leq \frac{1}{4\left(\lambda_0^{\varepsilon}\right)^2}e^{2\lambda_0^{\varepsilon}T}\int_0^T \zeta(s)^2\,ds.
\end{align*}
Besides, from the definition of $\zeta(t)$, and since $f$ is localized in $\omega$, it immediately follows
\begin{align*}
	\left|\zeta(t)\right|^2 \leq \norm{f(t)}{L^2(\Omega)}^2\norm{\phi_0^{\varepsilon}}{L^2(\omega)}^2.
\end{align*}
Hence, we deduce from \eqref{ode est} that
\begin{align*}
	\frac{1}{4\lambda_0^{\varepsilon}}\left(e^{2\lambda_0^{\varepsilon}T}-1\right) \leq \int_Q u^2\,dxdt + \frac{\norm{\phi_0^{\varepsilon}}{L^2(\omega)}^2}{4\left(\lambda_0^{\varepsilon}\right)^2}e^{2\lambda_0^{\varepsilon}T}\int_0^T \norm{f(t)}{L^2(\Omega)}^2\,dt,
\end{align*}
that implies either
\begin{align*}
	\frac{1}{8\lambda_0^{\varepsilon}}\left(e^{2\lambda_0^{\varepsilon}T}-1\right) \leq \int_Q u^2\,dxdt 
\end{align*}
or
\begin{align*}
	\frac{1}{8\lambda_0^{\varepsilon}}\left(e^{2\lambda_0^{\varepsilon}T}-1\right) \leq \frac{\norm{\phi_0^{\varepsilon}}{L^2(\omega)}^2}{4\left(\lambda_0^{\varepsilon}\right)^2}e^{2\lambda_0^{\varepsilon}T}\int_0^T \norm{f(t)}{L^2(\Omega)}^2\,dt.
\end{align*}
In any case, for any $f\in L^2((0,T);L^2(\Omega))$ with support in $\omega$ we get
\begin{align*}
	J_{u_0^{\varepsilon}}^{\varepsilon}(f) \geq \inf\left\{\frac{e^{2\lambda_0^{\varepsilon}T}-1}{16\lambda_0^{\varepsilon}},\; \frac{\lambda_0^{\varepsilon}}{4\norm{\phi_0^{\varepsilon}}{L^2(\omega)}^2}\left(1-e^{2\lambda_0^{\varepsilon}T}\right)\right\}.
\end{align*}
This last bound blows up as $\varepsilon\to 0^+$, due to the estimates \eqref{eigenv behaviour} and \eqref{eigenf behaviour}. Indeed, by definition of $\omega$, we can find $\beta>0$ such that $\omega\subset\Omega\setminus\Omega_{\beta}$ and therefore 
\begin{align*}
	\norm{\phi_0^{\varepsilon}}{L^2(\omega)} \leq \norm{\phi_0^{\varepsilon}}{L^2(\Omega\setminus\Omega_{\beta})} \leq \norm{\phi_0^{\varepsilon}}{H^1(\Omega\setminus\Omega_{\beta})} \to 0,
\end{align*}
as $\varepsilon\to 0^+$. This concludes the proof.
\end{proof}
%section 6
\section{Proof of the Carleman estimate}\label{carleman_proof}
Before giving the proof of the Carleman estimate \eqref{carleman}, it is important to remark that, in principle, the solutions of \eqref{heat hardy adj} do not have enough regularity to justify the computations; in particular, the $H^2$ regularity in the space variable that would be required for applying standard integration by parts may not be guaranteed. For this reason, we need to add some regularisation argument. 
\\
\indent In our case, this can be done by regularising the potential, i.e. by considering, instead of the operator $\mathcal{A}$ defined in \eqref{singular op}, the following 
\begin{align}\label{operator A reg}
	\mathcal{A}_n  v:=\Delta v+\frac{\mu_1}{(\delta+1/n)^2}\,v,\;\;\;\;n>0.
\end{align}
$\indent$ The domain of this new operator is $\mathcal{D}(\mathcal{A}_n)=\mathcal{D}(-\Delta)=H_0^1(\Omega)\cap H^2(\Omega)$, due to the fact that now our potential is bounded on $\Omega$, and the solution $v_n$ of the corresponding parabolic equation possess all the regularity needed to justify the computations. Passing to the limit as $n\to +\infty$, we can then recover our result for the solution $v$ of \eqref{heat hardy adj}.
\\
\indent In order to simplify our presentation, we will skip this regularisation process and we will write directly the formal computations for the solution of \eqref{heat hardy adj}. Moreover, we are going to present here the main ideas of the proof of the inequality, using some some technical Lemmas, which will be proved in appendix A.
\subsubsection*{Step 1. Notation and rewriting of the problem}
For any solution $v$ of the adjoint problem \eqref{heat hardy adj}, and for any $R>0$, we define
\begin{align}\label{zeta}
	z(x,t):=v(x,t)e^{-R\sigma(x,t)},
\end{align}
which satisfies 
\begin{align}\label{bc time}
	z(x,0)=z(x,T)=0
\end{align}
in $H_0^1(\Omega)$, due to the definition of $\sigma$. The positive parameter $R$ is meant to be large. Plugging $v(x,t)=z(x,t)e^{R\sigma(x,t)}$ in \eqref{heat hardy adj}, we obtain that $z$ satisfies 
\begin{align}\label{heat z}
	\begin{array}{ll}
		\displaystyle z_t+\Delta z+\frac{\mu}{\delta^2}z+2R\nabla z\cdot\nabla\sigma+Rz\Delta\sigma+z\left(R\sigma_t+R^2|\nabla\sigma|^2\right)=0, & (x,t)\in\Omega\times (0,T)
	\end{array}
\end{align}
with boundary conditions
\begin{align}\label{heat z bc}
	\begin{array}{ll}
		z(x,t)=0, & (x,t)\in\Gamma\times(0,T). 
	\end{array}
\end{align}
Next, we define a smooth positive function $\alpha(x)$ such that
\begin{align}\label{alpha}
	\alpha(x)=\left\{\begin{array}{ll}
		0 & x\in\Omega_{r_0/2}
		\\ 1 & x\in\Omega\setminus\Omega_{r_0}
	\end{array}\right.
\end{align} 
where $\Omega_{r_0}$ has been introduced in \eqref{notation set}. Setting
\begin{align*}
	\mathrm{S}z := \Delta z+\frac{\mu}{\delta^2}z+z\left(R\sigma_t+R^2|\nabla\sigma|^2\right), \;\; \mathrm{A}z := z_t+2R\nabla z\cdot\nabla\sigma+Rz\Delta\sigma(1+\alpha), \;\; \mathrm{P}z := -R\alpha z\Delta\sigma,
\end{align*}	
one easily deduce from \eqref{heat z} that
\begin{align*}
	\mathrm{S}z+\mathrm{A}z+\mathrm{P}z=0, & \;\;\;\;\norm{\mathrm{S}z}{L^2(Q)}^2+\norm{\mathrm{A}z}{L^2(Q)}^2+2\langle\mathrm{S},\mathrm{A}\rangle_{L^2(Q)} =\norm{\mathrm{P}z}{L^2(Q)}^2.
\end{align*}
In particular, we obtain that the quantity
\begin{align}\label{I}
	I=\langle\mathrm{S}z,\mathrm{A}z\rangle_{L^2(Q)}-\frac{1}{2}\norm{R\alpha z\Delta\sigma}{L^2(Q)}^2
\end{align}
is not positive.

\subsubsection*{Step 2. Computation of the scalar product}
\begin{lemma}\label{scalar product lemma}
The following identity holds:
\begin{align}\label{scalar product}
	\nonumber I &= R\int_Q |\partial_nz|^2\partial_n\sigma\,dxdt - 2R\int_Q D^2\sigma(\nabla z,\nabla z)\,dxdt - R\int_Q \alpha\Delta\sigma|\nabla z|^2\,dxdt 
	\\
	\nonumber &+ R\int_Q \left(\nabla(\Delta\sigma)\cdot\nabla\alpha\right)z^2\,dxdt + \frac{R}{2}\int_Q \Delta\sigma\Delta\alpha\, z^2\,dxdt + R\mu\int_Q \alpha\Delta\sigma\frac{z^2}{\delta^2}\,dxdt 
	\\
	\nonumber &+ 2R\mu\int_Q \left(\nabla\delta\cdot\nabla\sigma\right)\frac{z^2}{\delta^3}\,dxdt + \frac{R}{2}\int_Q \Delta^2\sigma(1+\alpha)z^2\,dxdt - 2R^3\int_Q D^2\sigma(\nabla\sigma,\nabla\sigma)z^2\,dxdt 
	\\
	\nonumber &+ R^3\int_Q \alpha\Delta\sigma|\nabla\sigma|^2z^2\,dxdt - \frac{R^2}{2}\int_Q \alpha^2|\Delta\sigma|^2z^2\,dxdt -\frac{1}{2}\int_Q \left(R\sigma_{tt}+2R^2(|\nabla\sigma|^2)_t\right)z^2\,dxdt 
	\\
	&+ R^2\int_Q \alpha\sigma_t\Delta\sigma\, z^2\,dxdt.
\end{align}
\end{lemma}
The proof of Lemma \ref{scalar product lemma} will be presented in the appendix. Moreover, in what follows we will split \eqref{scalar product} in four parts; first of all, let us define the boundary term
\begin{align}\label{I_bd}
	& I_{bd} = R\int_{\Sigma} |\partial_nz|^2\partial_n\sigma\,dsdt,
\end{align}
where $\Sigma:=\partial\Omega\times(0,T).$ 
\\
\indent Secondly, we define $I_l$ as the sum of the integrals linear in $\sigma$ which do not involve any time derivative
\begin{align}\label{I_l}
	\nonumber I_{l} &= -2R\int_Q D^2\sigma(\nabla z,\nabla z)\,dxdt - R\int_Q \alpha\Delta\sigma|\nabla z|^2\,dxdt + R\int_Q \left(\nabla(\Delta\sigma)\cdot\nabla\alpha\right)z^2\,dxdt 
	\\ \nonumber
	&+ \frac{R}{2}\int_Q \Delta\sigma\Delta\alpha\, z^2\,dxdt + R\mu\int_Q \alpha\Delta\sigma\frac{z^2}{\delta^2}\,dxdt 
	\\ 
	&+ 2R\mu\int_Q \left(\nabla\delta\cdot\nabla\sigma\right)\frac{z^2}{\delta^3}\,dxdt + \frac{R}{2}\int_Q \Delta^2\sigma(1+\alpha)z^2\,dxdt.
\end{align}
\indent Then, we consider the sum of the integrals involving non-linear terms in $\sigma$ and without any time derivative, that is
\begin{align}\label{I_nl}
	& I_{nl} = -2R^3\int_Q D^2\sigma(\nabla\sigma,\nabla\sigma)z^2\,dxdt + R^3\int_Q \alpha\Delta\sigma|\nabla\sigma|^2z^2\,dxdt - \frac{R^2}{2}\int_Q \alpha^2|\Delta\sigma|^2z^2\,dxdt.
\end{align}
Finally, we define the terms involving the time derivative in $\sigma$ as
\begin{align}\label{I_t}
	& I_t = -\frac{1}{2}\int_Q \left(R\sigma_{tt}+2R^2(|\nabla\sigma|^2)_t\right)z^2\,dxdt +   R^2\int_Q \alpha\sigma_t\Delta\sigma\, z^2\,dxdt.
\end{align}

\subsubsection*{Step 3. Bounds for the quantities $I_b$, $I_l$, $I_{nl}$ and $I_t$}
We now estimates the four quantities \eqref{I_bd}, \eqref{I_l}, \eqref{I_nl} and \eqref{I_t} separately.  
\begin{lemma}\label{I_bd lemma}
It holds that $I_{bd}=0$ for any $\lambda>1$ 
\end{lemma}

\begin{lemma}\label{I_l lemma}
There exists $\lambda_0$ such that for any $\lambda\geq\lambda_0$ and any $R>0$, and for any $r_0$ as in \eqref{r_0}, it holds 
\begin{align}\label{I_l bound}
	\nonumber I_l &\geq B_1R\int_Q \theta\left(\delta^{2-\gamma}|\nabla z|^2+\frac{z^2}{\delta^{\gamma}}\right)\,dxdt + \frac{\lambda R}{2}\intr{\Omega_{r_0}\times(0,T)} \theta\left(\frac{\delta}{r_0}\right)^{\lambda-2}|\nabla z|^2\,dxdt 
	\\
	& - B_2\lambda^2R\intr{\omega_0\times(0,T)} \theta\left(\frac{\delta}{r_0}\right)^{\lambda}\phi|\nabla z|^2\,dxdt + B_3\lambda^2R\intr{\mathcal{O}\times(0,T)} \theta\left(\frac{\delta}{r_0}\right)^{\lambda}\phi|\nabla z|^2\,dxdt - B_{\lambda}R\int_Q \theta z^2\,dxdt,
\end{align}
where $B_1$, $B_2$ and $B_3$ are positive constants independent on $R$ and $\lambda$, and $B_{\lambda}$ is a positive constant independent on $R$.
\end{lemma}

\begin{lemma}\label{I_nl lemma}
There exists $\lambda_0$ such that for any $\lambda\geq\lambda_0$ there exists $R_0=R_0(\lambda)$ such that for any $R\geq R_0$ and for any $r_0$ as in \eqref{r_0} it holds 
\begin{align}\label{I_nl bound}
	\nonumber I_{nl} &\geq \frac{R^3}{2}\intr{\Omega_{r_0}\times(0,T)}\theta^3\delta^2z^2\,dxdt +  B_5\lambda^4R^3\intr{\mathcal{O}\times(0,T)} \theta^3\left(\frac{\delta}{r_0}\right)^{3\lambda}\phi^3z^2\,dxdt 
	\\
	& - B_6\lambda^4R^3\intr{\omega_0\times(0,T)}  \theta^3\left(\frac{\delta}{r_0}\right)^{3\lambda}\phi^3z^2\,dxdt,
\end{align}
for some positive constants $B_5$ and $B_6$ uniform in $R$ and $\lambda$.
\end{lemma}

Taking into account the negative terms in the expression of $I_l$ that we want to get rid of, we define
\begin{align}\label{I_r}
	I_r=I_t-B_{\lambda}R\int_Q \theta z^2\,dxdt.
\end{align}

\begin{lemma}\label{I_r lemma}
There exists $\lambda_0$ such that for any $\lambda\geq\lambda_0$ there exists $R_0=R_0(\lambda)$ such that for any $R\geq R_0$ and for any $r_0$ as in \eqref{r_0} it holds 
\begin{align}\label{I_r bound}
	|I_r| &\leq \frac{B_1}{2}R \int_Q\theta\frac{z^2}{\delta^{\gamma}}\,dxdt +  \frac{B_5}{2}\lambda^4R^3\intr{\mathcal{O}\times(0,T)} \theta^3\left(\frac{\delta}{r_0}\right)^{3\lambda}\phi^3z^2\,dxdt + \frac{R^3}{4}\intr{\Omega_{r_0}\times(0,T)}  \theta^3\delta^2z^2\,dxdt,
\end{align}	
where $B_1$ and $B_5$ are the positive constants introduced in Lemmas \ref{I_l lemma} and \ref{I_nl lemma}, respectively.
\end{lemma}
\noindent The proofs of Lemmas \ref{I_bd lemma}, \ref{I_l lemma}, \ref{I_nl lemma} and \ref{I_r lemma} will be presented again in the appendix.
\subsubsection*{Step 4. Conclusion}
\noindent From the Lemmas above, we obtain the Carleman estimates in the variable $z$ as follows

\begin{theorem}\label{carleman z thm}
There exist two positive constants $\lambda_0$ and $\mathcal{L}$ such that for any $\lambda\geq\lambda_0$ there exists $R_0=R_0(\lambda)$ such that for any $R\geq R_0$ it holds 
\begin{align}\label{carleman z}
	\nonumber R\int_Q & \theta\left(\delta^{2-\gamma}|\nabla z|^2+\frac{1}{2}\frac{z^2}{\delta^{\gamma}}\right)\,dxdt + \lambda R\intr{\Omega_{r_0}\times(0,T)} \theta\left(\frac{\delta}{r_0}\right)^{\lambda-2}|\nabla z|^2\,dxdt  + \lambda^2R\intr{\mathcal{O}\times(0,T)} \theta\left(\frac{\delta}{r_0}\right)^{\lambda}\phi|\nabla z|^2\,dxdt 
	\\
	\nonumber &+ R^3\intr{\Omega_{r_0}\times(0,T)}\theta^3\delta^2z^2\,dxdt +  \lambda^4R^3\intr{\mathcal{O}\times(0,T)} \theta^3\left(\frac{\delta}{r_0}\right)^{3\lambda}\phi^3z^2\,dxdt 
	\\
	\nonumber &\leq \mathcal{L}\left(\lambda^4R^3\intr{\omega_0\times(0,T)}  \theta^3\left(\frac{\delta}{r_0}\right)^{3\lambda}\phi^3z^2\,dxdt + \lambda^2R\intr{\omega_0\times(0,T)} \theta\left(\frac{\delta}{r_0}\right)^{\lambda}\phi|\nabla z|^2\,dxdt\right)
	\\
\end{align}
\end{theorem}
\noindent Coming back from the variable $z$ to the solution $v$ of \eqref{heat hardy adj}, we finally obtain Theorem \ref{carleman thm}.
%section 7
\section{Open problems and perspectives}\label{open pb} 
\noindent We conclude this paper with some open problem and perspective related to our work. 
\begin{itemize}
\item[•] \textit{Boundary controllability}. In this article it is treated the controllability problem for the equation 
\begin{align}\label{open pb eq}
	& u_t-\Delta u-\frac{\mu}{\delta^2}u=0, \;\;\;\; (x,t)\in\Omega\times (0,T)
\end{align}
with a distributed control located in an open set $\omega\subset\Omega$. An immediate and interesting extension of the result we obtained, would be the analysis of boundary controllability for equation \eqref{open pb eq}. In this framework, a first approach to this problem in one space dimension is given in \cite{biccari2015boundary}, where the author is able to obtain boundary controllability for a heat equation with an inverse-square potential presenting singularities all-over the boundary. The multi-dimensional case, instead, remains at the moment unaddressed. As it is explained in \cite{biccari2015boundary}, the main difficulty of this problem is to understand the behaviour of the normal derivative of the solution when approaching the boundary. Indeed, due to the presence of the singularity this normal derivative degenerates and this degeneracy would need to be properly compensated, in order to build the control for our equation. More in details, always referring to \cite{biccari2015boundary}, we believe that we need to introduce a weighted normal derivative in the form $\delta^{\alpha}\partial_{\nu}u$, with a coefficient $\alpha$ which has to be identified. Then, the weight $\sigma$ we employ in our Carleman has to be modified accordingly; we propose $\tilde{\sigma}(x,t)=\theta(t)(C_{\lambda}+\delta^{1+2\alpha}\psi-(\delta/r_0)^{\lambda}\phi)$, with $\theta$ and $\psi$ as in \eqref{sigma}, since this function would allow to obtain the weighted normal derivative we mentioned above in the boundary term of the Carleman inequality. The main difficulty would then be to show that, with this choice of the weight, it is possible to obtain suitable bounds for the distributed terms that shall lead to the inequality we seek.

\item[•] \textit{Wave equation.} It would be interesting to investigate controllability properties also for a wave equation with singular inverse-square potential of the type $\mu/\delta^2$. Even if there are already results in the literature on this topic (see, for instance \cite{cazacu2012schrodinger} and \cite{vancostenoble2009hardy}), in our knowledge nobody treated the case of a potential with singularities arising all over the boundary. This is a very challenging issue; indeed, already in the one dimensional case, the presence of the singularity all over the boundary makes the multiplier approach extremely tricky, in the sense that is very difficult to identify, if possible, the correct multiplier for obtaining a Pohozaev identity. On the other hand, this would be surely a problem which deserves a more deep analysis.

\item[•] \textit{Optimality of our results.} In the definition of the weight $\sigma$ we consider an exponent $k=3$ for our function $\theta$; the motivation of this choice is that for lower exponents we are not able to bound some terms in our Carleman inequality. However, this has consequences on the cost of the control as the time tends to zero (see, for instance, \cite{ervedoza2010systematic}, \cite{miller2006control}), which is not of the order of $\exp(C/T)$, as expected for the heat equation, but rather of $\exp(C/T^3)$. Therefore, it would be interesting to reduce the exponent in the definition of $\theta$ up to $k=1$ and try to obtain a Carleman estimate with this new choice for the weight.  
\end{itemize}
% appendix
\appendix
\renewcommand{\thesection}{\Alph{section}}
\section{Proof of technical Lemmas}
The computations for obtaining the Carleman estimate are very long; in order to simplify the presentation, in Section \ref{carleman_proof} we divided these computations in four step and we introduced several preliminary results, Lemmas \ref{scalar product lemma} to \ref{I_r lemma}. We present now the proof of these Lemmas. 
\\
At this purpose, we remind that the distance function $\delta$ satisfies the following properties
\begin{subequations}
\label{eq:tot}
\begin{empheq}{align}
		& \delta\in C^{0,1}(\overline{\Omega}), \\
		& |\nabla u| = 1, \, \textrm{ a.e. in }\, \Omega, \label{grad_distance} \\
		& \textrm{there exists a constant } P>0 \textrm{ such that } |\Delta\delta|\leq P/\delta, \textrm{ a.e. in } \Omega. \label{delta_distance}
\end{empheq}
\end{subequations}
Furthermore, we are going to need the following result
\begin{lemma}\label{lemma_psi}
Assume that $\psi$ is the function defined in \eqref{psi} by means of $\psi_1$ and $\varpi$. Then, there exists a constant $\Do{\psi_1}>0$, which depends on $\psi_1$, such that 
\begin{align}\label{psi rel}
	& |\nabla\delta\cdot\nabla\psi(x)-\varpi\psi_1(x)|\leq\varpi\Do{\psi_1}.
\end{align} 
\end{lemma}
\begin{proof}
By definition of $\psi$ and Cauchy-Scwarz inequality, using \eqref{grad_distance} and since $\psi_1$ is bounded,  we immediately have 
\begin{align*}
	|\nabla\delta\cdot\nabla\psi(x)-\varpi\psi_1(x)| = \varpi|\nabla\delta\cdot\nabla\psi_1(x)-\psi_1(x)|\leq\varpi|\nabla\psi_1-\psi_1|\leq\varpi\Do{\psi_1}.
\end{align*}
\end{proof}
\noindent Now, for $\sigma$ as in \eqref{sigma} we introduce the notations
\begin{align*} 
	&\sigma_{\delta} =-\theta\tau_{\delta}=-\theta\delta^2\psi, \;\;\;\;\; \sigma_{\phi} =-\theta\tau_{\phi}=-\theta\left(\frac{\delta}{r_0}\right)^{\lambda}\phi, \;\;\;\;\; \tau =\tau_{\delta}+\tau_{\phi},
\end{align*}
so that $\sigma(x,t)=C_{\lambda}\theta(t)+\sigma_{\delta}(x,t)+\sigma_{\phi}(x,t).$ Next, we deduce some formulas for $\tau_{\delta}$ and $\tau_{\phi}$ that we are going to use later in our computations. More precisely, for all $x\in\RR^N$ and any $i,j\in\{1,\ldots,N\}$ we have
\begin{align}
	& \partial_{x_i}\tau_{\delta}=2\psi\delta\delta_{x_i} + \delta^2\psi_{x_i}, \label{deriv delta 1}
	\\
	& \partial^2_{x_ix_j}\tau_{\delta} = 2\psi\delta_{x_i}\delta_{x_j} +  2\delta(\psi_{x_j}\delta_{x_i} + \psi\delta_{x_ix_j}) + 2\delta\psi_{x_i}\delta_{x_j} + \delta^2\psi_{x_ix_j} \label{deriv delta 2}
\end{align} 
and
\begin{align}
	& \Delta\tau_{\delta} = 2\psi + 4\delta(\nabla\delta\cdot\nabla\psi) + 2\delta\psi\Delta\delta+\delta^2\Delta\psi, \label{deriv delta 3}
	\\
	& D^2\tau_{\delta}(\xi,\xi) = 2\psi(\xi\cdot\nabla\delta)^2 + 2\delta\psi D^2\delta(\xi,\xi) + 4\delta(\xi\cdot\nabla\delta)(\xi\cdot\nabla\psi) + \delta^2D^2\psi(\xi,\xi), \;\; \forall\xi\in\RR^N. \label{deriv delta 4}
\end{align}
On the other hand
\begin{align}
	& \partial_{x_i}\tau_{\phi} = \frac{\phi}{r_0^{\lambda}}(\lambda\delta^{\lambda-1}\delta_{x_i} + \lambda\delta^{\lambda}\psi_{x_i}), \label{deriv phi 1}
	\\
	& \partial^2_{x_ix_j}\tau_{\phi} = \frac{\phi}{r_0^{\lambda}}\Big(\lambda(\lambda-1)\delta^{\lambda-2}\delta_{x_i}\delta_{x_j} + \lambda\delta^{\lambda-1}\delta_{x_ix_j} + \lambda^2\delta^{\lambda-1}(\psi_{x_j}\delta_{x_i}+\psi_{x_i}\delta_{x_j}) + \lambda\delta^{\lambda}\psi_{x_ix_j} + \lambda^2\delta^{\lambda}\psi_{x_i}\psi_{x_j}\Big) \label{deriv phi 2}
\end{align} 
and
\begin{align}
	\Delta\tau_{\phi} = \frac{\phi}{r_0^{\lambda}}&\Big(\lambda(\lambda-1)\delta^{\lambda-2} + \lambda\delta^{\lambda-1}\Delta\delta + 2\lambda^2\delta^{\lambda-1}(\nabla\delta\cdot\nabla\psi) + \lambda\delta^{\lambda}\Delta\psi + \lambda^2\delta^{\lambda}|\nabla\psi|^2\Big), \label{deriv phi 3}
	\\
	\nonumber D^2 \tau_{\phi}(\xi,\xi) &= \frac{\phi}{r_0^{\lambda}}\Big(\lambda(\lambda-1)\delta^{\lambda-2}(\xi\cdot\nabla\delta)^2 + \lambda\delta^{\lambda-1}D^2\delta(\xi,\xi) + 2\lambda^2\delta^{\lambda-1}(\xi\cdot\nabla\delta)(\xi\cdot\nabla\psi) 
	\\
	& + \lambda\delta^{\lambda}D^2\psi(\xi,\xi) + \lambda^2\delta^{\lambda}(\xi\cdot\nabla\psi)^2\Big), \;\;\;\forall\xi\in\RR^N. \label{deriv phi 4}
\end{align}

\subsection*{Upper and lower bounds for $\Delta\tau_{\delta}$, $\Delta\tau_{\phi}$, $D^2\tau_{\delta}(\xi,\xi)$ and $D^2\tau_{\phi}(\xi,\xi)$ }

\begin{proposition}\label{delta estimates}
For $r_0$ as in \eqref{r_0} we have
\begin{align}
	& \Delta\tau_{\delta}\geq 0, D^2\tau_{\delta}\geq 0, & &\forall x\in\Omega_{r_0},\,\forall\xi\in\RR^N, \label{delta estimates 1}
	\\
	& |D^2\tau_{\delta}(\xi,\xi)|\leq C_1|\xi|^2, & &\forall x\in\Omega,\,\forall\xi\in\RR^N, \label{delta estimates 2}
	\\
	& |\Delta\tau_{\delta}|\leq C_2, & &\forall x\in\Omega_{r_0}. \label{delta estimates 3}
\end{align}
where $C_1$ and $C_2$ are constants depending on $\Omega$ and $\psi$.
\end{proposition}

\begin{proposition}\label{phi estimates}
For $r_0$ and $\varpi$ as in \eqref{r_0} and \eqref{varpi} we have
\begin{align}
	& D^2\tau_{\phi}\geq \frac{\lambda}{2}\left(\frac{\delta}{r_0}\right)^{\lambda-2}\phi|\xi|^2, & & \forall x\in\Omega_{r_0},\,\forall\xi\in\RR^N, \label{phi estimates 1}
	\\
	& \Delta\tau_{\phi}\geq\lambda^2\left(\frac{\delta}{r_0}\right)^{\lambda}\phi, & & \forall x\in\mathcal{O}, \label{phi estimates 2}
	\\
	& D^2\tau_{\phi}\geq -\lambda C_3\left(\frac{\delta}{r_0}\right)^{\lambda-2}\phi|\xi|^2, & & \forall x\in\Omega,\,\forall\xi\in\RR^N, \label{phi estimates 3}
\end{align}
for $\lambda$ large enough, where $C_3$ is a constant depending on $\Omega$, $r_0$ and $\psi$.
\end{proposition}

\begin{proof}[Proof of Proposition \ref{delta estimates}]
Observe that the proofs of \eqref{delta estimates 2} and \eqref{delta estimates 3} are trivial. To prove \eqref{delta estimates 1}, instead, it is enough to show that $D^2\tau_{\delta}(\xi,\xi)\geq 0$ in $\Omega_{r_0}$ since this also implies that $\Delta\tau_{\delta}\geq 0$ in $\Omega_{r_0}$, simply choosing $\xi=e_i$ for all $i\in\{1,\ldots,N\}$. Now, we have that, for $x\in\Omega_{r_0}$
\begin{align}\label{dist bd}
	\delta(x)=|x-\textrm{pr}(x)|
\end{align}
where $\textrm{pr}(x)$ is the projection of $x$ on $\Gamma$. Hence \eqref{deriv delta 4} becomes
\begin{align*}
	D^2\tau_{\delta}(\xi,\xi) = 2\psi|\xi|^2 + 4\left(\xi\cdot\Big(x-\textrm{pr}(x)\Big)\right)(\xi\cdot\nabla\psi) + \delta^2D^2\psi(\xi,\xi), \;\;\;\forall\xi\in\RR^N.
\end{align*}
Now, using Cauchy-Scwarz inequality we obtain
\begin{align*}
	D^2\tau_{\delta}(\xi,\xi) &\geq (2\psi - 4\delta|D\psi|_{\infty} - \delta^2|D^2\psi|_{\infty})|\xi|^2 \geq (2\psi - r_0(4|D\psi|_{\infty} + |D^2\psi|_{\infty}))|\xi|^2 \geq 0.
\end{align*}
since $r_0$ satisfies \eqref{r_0}.
\end{proof}

\begin{proof}[Proof of Proposition \ref{phi estimates}]
First of all, we rewrite \eqref{deriv phi 4} as $D^2\tau_{\phi}(\xi,\xi)=\phi(1/r_0)^{\lambda}\,\mathcal{S}_{\phi}$, where 
\begin{align}\label{s phi}
	\nonumber\mathcal{S}_{\phi} &= \lambda(\lambda-1)\delta^{\lambda-2}(\xi\cdot\nabla\delta)^2 + \lambda\delta^{\lambda-1}D^2\delta(\xi,\xi) + 2\lambda^2\delta^{\lambda-1}(\xi\cdot\nabla\delta)(\xi\cdot\nabla\psi) 
	\\
	&+ \lambda\delta^{\lambda}D^2\psi(\xi,\xi) + \lambda^2\delta^{\lambda}(\xi\cdot\nabla\psi)^2.
\end{align}
Next, we have 
\begin{align*}
	|2\lambda^2\delta^{\lambda-1}(\xi\cdot\nabla\delta)(\xi\cdot\nabla\psi)|\leq a\lambda^2\delta^{\lambda-2}(\xi\cdot\nabla\delta)^2+\frac{\lambda^2}{a}\delta^{\lambda}(\xi\cdot\nabla\psi)^2, \;\;\forall a>0,
\end{align*}
which combined with \eqref{s phi} leads to
\begin{align*}
	\mathcal{S}_{\phi} \geq (\lambda^2-\lambda-a\lambda^2)\delta^{\lambda-2}(\xi\cdot\nabla\delta)^2 + \lambda\delta^{\lambda-1}D^2\delta(\xi,\xi) + \lambda\delta^{\lambda}D^2\psi(\xi,\xi) + \left(\lambda^2-\frac{\lambda^2}{a}\right)\delta^{\lambda}(\xi\cdot\nabla\psi)^2.
\end{align*}
Choosing now $a$ such that $\lambda^2(1-a)-\lambda=0$, i.e. $a=(\lambda-1)/\lambda$, we have
\begin{align}\label{s phi est}
	\mathcal{S}_{\phi} \geq \lambda\delta^{\lambda-1}D^2\delta(\xi,\xi) + \lambda\delta^{\lambda}D^2\psi(\xi,\xi) - -\frac{\lambda^2}{\lambda-1}\delta^{\lambda}|\nabla\psi|^2|\xi|^2.
\end{align}
Applying \eqref{s phi est} for $x\in\Omega_{r_0}$ we deduce
\begin{align*}
	\mathcal{S}_{\phi} & \geq \frac{\lambda}{2}\delta^{\lambda-2}|\xi|^2 + \lambda\delta^{\lambda-2}|\xi|^2\left(\frac{1}{2}-\frac{\lambda}{\lambda-1}\delta^2|D\psi|_{\infty}^2-\delta^2|D^2\psi|_{\infty}\right)
	\\
	& \geq \frac{\lambda}{2}\delta^{\lambda-2}|\xi|^2 + \lambda\delta^{\lambda-2}|\xi|^2\left(\frac{1}{2}-r_0^2\Big(2|D\psi|_{\infty}^2+|D^2\psi|_{\infty}\Big)\right) \geq \frac{\lambda}{2}\delta^{\lambda-2}|\xi|^2, 
\end{align*}
for $r_0$ as in \eqref{r_0}. This immediately yields the proof of \eqref{phi estimates 1}.
\\
\indent Let us now prove \eqref{phi estimates 2}. According to Lemma \ref{lemma_psi}, to the definition of $\psi$ and to \eqref{delta_distance} and \eqref{deriv phi 3} we get
\begin{align*}
	\Delta\tau_{\phi} & \geq \frac{\phi}{r_0^{\lambda}}\left(\lambda(\lambda-1-P)\delta^{\lambda-2}+2\lambda^2\delta^{\lambda-1}(\varpi\psi_1-\varpi\Do{\psi_1})+\lambda\delta^{\lambda}\Delta\psi+\lambda^2\delta^{\lambda}|\nabla\psi|^2\right)
	\\
	& \geq \lambda^2\left(\frac{\delta}{r_0}\right)^{\lambda}\phi\left(|\nabla\psi|^2-\frac{2\varpi\Do{\psi_1}}{r_0}-\varpi\frac{|\Delta\psi|}{\lambda}\right) \geq \lambda^2\left(\frac{\delta}{r_0}\right)^{\lambda}\phi\left(\varpi^2\varpi_0^2-\frac{2\varpi\Do{\psi_1}}{r_0}-\varpi\frac{|\Delta\psi|}{\lambda}\right)
	\\
	& \geq \lambda^2\left(\frac{\delta}{r_0}\right)^{\lambda}\phi
\end{align*}
for all $x\in\mathcal{O}$, if we take $\varpi$ as in \eqref{varpi} and $\lambda$ large enough.
\\
We conclude with the proof of \eqref{phi estimates 3}. From \eqref{deriv phi 4} for any $x\in\Omega$ we have
\begin{align*}
	D^2\tau_{\phi}(\xi,\xi) &= \frac{\phi}{r_0^{\lambda}}\left(\lambda^2\left(\delta^{\frac{\lambda}{2}-1}(\xi\cdot\nabla\delta)+\delta^{\frac{\lambda}{2}}(\xi\cdot\nabla\psi)\right)^2 + \lambda\delta^{\lambda-1}D^2\delta(\xi,\xi) + \lambda\delta^{\lambda}D^2\psi(\xi,\xi)-\lambda\delta^{\lambda-2}(\xi\cdot\nabla\delta)^2\right)
	\\
	& \geq \lambda\left(\frac{\delta}{r_0}\right)^{\lambda-2}\phi\left(\frac{1}{r_0^2}\left(\delta D^2\delta(\xi,\xi) + \delta^2D^2\psi(\xi,\xi)-(\xi\cdot\nabla\delta)^2\right)\right)
	\\
	& \geq -\lambda\left(\frac{\delta}{r_0}\right)^{\lambda-2}\phi\left(\frac{1}{r_0^2}\left(|D^2\delta|_{\infty} + R_{\Omega}^2|D^2\psi|_{\infty} + 1\right)\right)|\xi|^2,
\end{align*}
which gives us the validity of \eqref{phi estimates 3} for $C_3 = \left(|D^2\delta|_{\infty} + R_{\Omega}^2|D^2\psi|_{\infty} + 1\right)/r_0^2$.
\end{proof}

\subsection*{Bounds for $2D^2\tau(\nabla\tau,\nabla\tau)-\alpha\Delta\tau|\nabla\tau|^2$}

We provide here pointwise estimates for the quantity 
\begin{align*}
	2D^2\tau(\nabla\tau,\nabla\tau)-\alpha\Delta\tau|\nabla\tau|^2,
\end{align*}
which appears in the identity in Lemma \ref{scalar product lemma}.
\\
First of all, we have
\begin{align*}
	\partial_{x_i}\tau &= 2\psi\delta\delta_{x_i} + \delta^2\psi_{x_i} + \frac{\phi}{r_0^{\lambda}}(\lambda\delta^{\lambda-1}\delta_{x_i} + \lambda\delta^{\lambda}\psi_{x_i}),
	\\
	\partial_{x_ix_j}^2\tau &= 2\psi\delta_{x_i}\delta_{x_j} +  2\delta(\psi_{x_j}\delta_{x_i} + \psi\delta_{x_ix_j}) + 2\delta\psi_{x_i}\delta_{x_j} + \delta^2\psi_{x_ix_j} 
	\\
	&+ \frac{\phi}{r_0^{\lambda}}\Big(\lambda(\lambda-1)\delta^{\lambda-2}\delta_{x_i}\delta_{x_j} + \lambda\delta^{\lambda-1}\delta_{x_ix_j} + \lambda^2\delta^{\lambda-1}(\psi_{x_j}\delta_{x_i}+\psi_{x_i}\delta_{x_j}) + \lambda\delta^{\lambda}\psi_{x_ix_j} + \lambda^2\delta^{\lambda}\psi_{x_i}\psi_{x_j}\Big),
\end{align*}
and in consequence
\begin{align}
	\nonumber \Delta\tau = 2\psi + 4 & \delta(\nabla\delta\cdot\nabla\psi) + 2\psi\Delta\delta+\delta^2\Delta\psi 
	\\
	+ \frac{\phi}{r_0^{\lambda}} & \Big(\lambda(\lambda-1)\delta^{\lambda-2} + \lambda\delta^{\lambda-1}\Delta\delta + 2\lambda^2\delta^{\lambda-1}(\nabla\delta\cdot\nabla\psi) + \lambda\delta^{\lambda}\Delta\psi + \lambda^2\delta^{\lambda}|\nabla\psi|^2\Big), \label{delta tau}
	\\
	\nonumber D^2\tau(\nabla\tau,\nabla\tau) &= 2\psi(\nabla\tau\cdot\nabla\delta)^2 + 2\delta\psi D^2\delta(\nabla\tau,\nabla\tau) + 4\delta(\nabla\tau\cdot\nabla\delta)(\nabla\tau\cdot\nabla\psi) + \delta^2D^2\psi(\nabla\tau,\nabla\tau) 
	\\
	\nonumber &+ \frac{\phi}{r_0^{\lambda}}\Big(\lambda(\lambda-1)\delta^{\lambda-2}(\nabla\tau\cdot\nabla\delta)^2 + \lambda\delta^{\lambda-1}D^2\delta(\nabla\tau,\nabla\tau) + 2\lambda^2\delta^{\lambda-1}(\nabla\tau\cdot\nabla\delta)(\nabla\tau\cdot\nabla\psi) 
	\\
	& + \lambda\delta^{\lambda}D^2\psi(\nabla\tau,\nabla\tau) + \lambda^2\delta^{\lambda}(\nabla\tau\cdot\nabla\psi)^2\Big). \label{d tau}
\end{align}
Using the expressions above we obtain the following useful formulas 
\begin{align*}
	& (\nabla\delta\cdot\nabla\tau)^2 = |\nabla\tau|^2 + \left((\nabla\delta\cdot\nabla\psi)^2-|\nabla\psi|^2\right)\left(\delta^2+\lambda\frac{\phi}{r_0^{\lambda}}\delta^{\lambda}\right)^2,
	\\
	& (\nabla\delta\cdot\nabla\tau)(\nabla\psi\cdot\nabla\tau) = |\nabla\tau|^2(\nabla\delta\cdot\nabla\psi) 
	+ \left(|\nabla\psi|^2 - (\nabla\delta\cdot\nabla\psi)^2\right)\left(2\delta\psi+\lambda\frac{\phi}{r_0^{\lambda}}\delta^{\lambda-1}\right)\left(\delta^2+\lambda\frac{\phi}{r_0^{\lambda}}\delta^{\lambda}\right),
	\\
	& (\nabla\psi\cdot\nabla\tau)^2 = |\nabla\psi|^2|\nabla\tau|^2 + \left((\nabla\delta\cdot\nabla\psi)^2-|\nabla\psi|^2\right)\left(2\delta\psi+\lambda\frac{\phi}{r_0^{\lambda}}\delta^{\lambda-1}\right)^2,
\end{align*}
and we finally conclude
\begin{align*}
	2D^2\tau(\nabla\tau,\nabla\tau)-\alpha\Delta\tau|\nabla\tau|^2 = T_1+T_2+T_3,
\end{align*}
where
\begin{align}
	\nonumber T_1 &= 2\psi(2-\alpha)|\nabla\tau|^2 + 4\delta\psi D^2\delta(\nabla\tau,\nabla\tau) + 2\delta^2D^2\psi (\nabla\tau,\nabla\tau) + 4(2-\alpha)\delta(\nabla\delta\cdot\nabla\psi)|\nabla\tau|^2 
	\\
	&- 2\delta\psi\alpha\Delta\delta|\nabla\tau|^2 - \delta^2\alpha\Delta\psi|\nabla\tau|^2, \label{T1}
	\\
	\nonumber T_2 &= 4\left(|\nabla\psi|^2-(\nabla\delta\cdot\nabla\psi)^2\right)\left(\delta^2+\lambda\frac{\phi}{r_0^{\lambda}}\delta^{\lambda}\right)\left(5\delta^2\psi+\lambda(2-\psi)\frac{\phi}{r_0^{\lambda}}\delta^{\lambda}\right)
	\\
	&+ \frac{\phi}{r_0^{\lambda}}\left(|\nabla\psi|^2-(\nabla\delta\cdot\nabla\psi)^2\right)\left(2\lambda^3\delta^{3\lambda-2}\left(\frac{\phi}{r_0^{\lambda}}\right)^2+\lambda^2(8\psi(1-\psi)-2)\delta^{\lambda+2}+4\lambda^2\frac{\phi}{r_0^{\lambda}}\delta^{2\lambda}+2\lambda\delta^{\lambda+2}\right), \label{T2}
	\\
	\nonumber T_3 &= \frac{\phi}{r_0^{\lambda}}\left\{\left[(\lambda^2(2-\alpha)-\lambda(2-\alpha+\alpha\delta\Delta\delta))
	\delta^{\lambda-2}+2\lambda^2\delta^{\lambda-1}(2-\alpha)(\nabla\delta\cdot\nabla\psi)\right.\right.
	\\
	&+ \left.\left. \lambda^2\delta^{\lambda}(2-\alpha)|\nabla\psi|^2-\lambda\alpha\delta^{\lambda}\Delta\psi\right]|\nabla\tau|^2+2\lambda\delta^{\lambda-1}D^2\delta(\nabla\tau,\nabla\tau)+2\lambda\delta^{\lambda}D^2\psi(\nabla\tau,\nabla\tau)\right\}.\label{T3}
\end{align}

\begin{proposition}\label{T1 estimates}
For $r_0$ as in \eqref{r_0}, there exist two positive constants $D_1$ and $D_2$ depending on $\Omega$ and $\psi$ such that the term $T_1$ in \eqref{T1} satisfies
\begin{align}
	& T_1\geq|\nabla\tau|^2, & & \forall x\in\Omega_{r_0}, \label{T1 est 1}
	\\
	& T_1\geq -D_1|\nabla\tau|^2, & & \forall x\in\mathcal{O}, \label{T1 est 2}
	\\
	& |T_1|\leq D_2|\nabla\tau|^2, & & \forall x\in\omega_0. \label{T1 est 3}
\end{align}
\end{proposition}

\begin{proposition}\label{T2 estimates}
There exists $\lambda_0$ large enough such that, for any $\lambda\geq\lambda_0$ and $r_0$ as in \eqref{r_0},  the term $T_2$ in \eqref{T2} satisfies
\begin{align}
	& T_2\geq -\frac{\phi}{r_0^{\lambda}}|D\psi|_{\infty}^2\left(D_3\lambda^2\psi^2+D_4\lambda^2\right)\delta^{\lambda+2}, & & \forall x\in\Omega_{r_0}, \label{T2 est 1}
	\\
	& T_2\geq 0, & & \forall x\in\tilde{\mathcal{O}}, \label{T2 est 2}
\end{align}
for some positive constants $D_3$and $D_4$ depending on ...
\end{proposition}

\begin{proposition}\label{T3 estimates}
There exists $\lambda_0$ large enough such that, for any $\lambda\geq\lambda_0$ and $r_0$ and $\varpi$ as in \eqref{r_0} and \eqref{varpi},  the term $T_3$ in \eqref{T3} satisfies
\begin{align}
	& T_3\geq \lambda^2\left(\frac{\phi}{r_0^{\lambda}}\delta^{\lambda-2}+\left(\frac{\delta}{r_0}\right)^{\lambda}\phi\right)|\nabla\tau|^2, & & \forall x\in\Omega\setminus\overline{\omega_0}, \label{T3 est 1}
	\\
	& T_3\leq \lambda^2D_5\frac{\phi}{r_0^{\lambda}}\delta^{\lambda-2}|\nabla\tau|^2, & & \forall x\in\Omega, \label{T3 est 2}
\end{align}
for some positive constant $D_5$, not depending on $\lambda$.
\end{proposition}

\begin{proposition}\label{grad tau estimates}
For any $r_0$ and $\varpi$ as in \eqref{r_0} and \eqref{varpi} it holds
\begin{align}
	& |\nabla\tau|^2 \geq \delta^2, & & \forall x\in\Omega_{r_0}, \label{grad tau est 1}
	\\
	& |\nabla\tau|^2 \geq \lambda^2\left(\frac{\delta}{r_0}\right)^{2\lambda}\phi^2, & & \forall x\in\mathcal{O}, \label{grad tau est 2}
	\\
	& |\nabla\tau|^2 \leq \lambda^2 D_6\left(\frac{\delta}{r_0}\right)^{2\lambda}\phi^2, & & \forall x\in\omega_0, \label{grad tau est 3}
\end{align}
where $D_6$ is a positive constant depending only on $\Omega$ and $\psi$.
\end{proposition}

\begin{proof}[Proof of Proposition \ref{T1 estimates}]
The inequalities \eqref{T1 est 2} and \eqref{T1 est 3} are obvious. Hence, we only need to prove \eqref{T1 est 1}. Due to the definition of $\alpha$, to the properties of $\psi$ and to Lemma \ref{lemma_psi}, and using \eqref{dist bd}, we have (see also \cite[Prop. 3.4]{cazacu2014controllability})
\begin{align*}
	T_1 & \geq \left(2 - r_0^2(8\varpi\Do{\psi_1} + 3|D^2\psi|_{\infty})\right)|\nabla\tau|^2 \geq \left(2 - r_0^2\left(8\frac{\Do{\psi_1}}{\varpi_0}|D\psi|_{\infty} + 3|D^2\psi|_{\infty}\right)\right)|\nabla\tau|^2 \geq |\nabla\tau|^2,
\end{align*}
in $\Omega_{r_0}$, for $r_0$ as in \eqref{r_0}.
\end{proof}

\begin{proof}[Proof of Proposition \ref{T2 estimates}]
Due to Cuachy-Scwarz inequality, the term $|\nabla\psi|^2-(\nabla\delta\cdot\nabla\psi)^2$ in \eqref{T2} is positive; hence
\begin{align*}
	4\left(|\nabla\psi|^2-(\nabla\delta\cdot\nabla\psi)^2\right)\left(\delta^2+\lambda\frac{\phi}{r_0^{\lambda}}\delta^{\lambda}\right)\left(5\delta^2\psi+\lambda(2-\psi)\frac{\phi}{r_0^{\lambda}}\delta^{\lambda}\right) \geq 4 D_7\delta^2\left(5\delta^2\psi+\lambda(2-\psi)\frac{\phi}{r_0^{\lambda}}\delta^{\lambda}\right)
	\\
	\geq -4 D_7\lambda\psi\frac{\phi}{r_0^{\lambda}}\delta^{\lambda+2} \geq - D_8\lambda^2\frac{\phi}{r_0^{\lambda}}\delta^{\lambda+2}
\end{align*}
for $\lambda$ large enough. From this \eqref{T2 est 1} follows trivially.
\\
\indent Concerning \eqref{T2 est 2}, it is straightforward to check that the inequality holds for $\lambda$ large enough, since the term in $\lambda^3$ is positive and it dominates all the other terms far away from the boundary.
\end{proof}

\begin{proof}[Proof of Proposition \ref{T3 estimates}]
For $x\in\Omega_{r_0}$, due to \eqref{dist bd}, the proof is analogous to the one of \cite[Prop. 3.6]{cazacu2014controllability} and we omit it here. Therefore, let us assume now $x\in\tilde{\mathcal{O}}$. Due to the definition of $\alpha$, for $\lambda$ large enough we have
\begin{align*}
	 \lambda^2(2-\alpha)-\lambda(2-\alpha-\alpha\delta\Delta\delta)\geq\lambda^2.
\end{align*}
Hence, from Lemma \ref{lemma_psi} and from the properties of $\psi$, for $x\in\Omega\setminus\overline{\omega_0}$ we have
\begin{align*}
	T_3 & \geq \frac{\phi}{r_0^{\lambda}}\left(\lambda^2\delta^{\lambda-2} + 2\lambda^2\delta^{\lambda-1}(2-\alpha)(\varpi\psi_1-\varpi\Do{\psi_1}) + \lambda^2\delta^{\lambda}(2-\alpha)|\nabla\psi|^2 - \lambda\alpha\delta^{\lambda}|D^2\psi|_{\infty} \right.
	\\
	&\left. - 2\lambda\delta^{\lambda-2}|D^2\delta|_{\infty} - 2\lambda\delta^{\lambda}|D^2\psi|_{\infty} \right)|\nabla\tau|^2
	\\
	& \geq \lambda^2\frac{\phi}{r_0^{\lambda}}\delta^{\lambda-2}|\nabla\tau|^2 + \lambda^2\frac{\phi}{r_0^{\lambda}}\delta^{\lambda}\left(\varpi^2|\nabla\psi_1|^2 - \frac{2\varpi\Do{\psi_1}}{\delta} - \frac{2+\alpha}{\lambda}|D^2\psi|_{\infty} - 2\frac{|D^2\delta|_{\infty}}{\delta^2\lambda}\right)|\nabla\tau|^2
	\\
	& \geq \lambda^2\frac{\phi}{r_0^{\lambda}}\delta^{\lambda-2}|\nabla\tau|^2 + \lambda^2\frac{\phi}{r_0^{\lambda}}\delta^{\lambda}\left(\varpi^2|\nabla\psi_1|^2 - \frac{2\varpi\Do{\psi_1}}{r_0} - \frac{2+\alpha}{\lambda}|D^2\psi|_{\infty} - 2\frac{|D^2\delta|_{\infty}}{r_0^2\lambda}\right)|\nabla\tau|^2
	\\
	& \geq \lambda^2\frac{\phi}{r_0^{\lambda}}\delta^{\lambda-2}|\nabla\tau|^2 + \lambda^2\frac{\phi}{r_0^{\lambda}}\delta^{\lambda}\left(\frac{\varpi^2\varpi_0^2}{2} - \frac{2\varpi\Do{\psi_1}}{r_0}\right)|\nabla\tau|^2 \geq \lambda^2\frac{\phi}{r_0^{\lambda}}\delta^{\lambda-2}|\nabla\tau|^2 + \lambda^2\frac{\phi}{r_0^{\lambda}}\delta^{\lambda}|\nabla\tau|^2,
\end{align*}
for $\lambda$ large enough and $\varpi$ as in \eqref{varpi}. Concerning \eqref{T3 est 2}, once again the proof is trivial and we omit it here.
\end{proof}

\begin{proof}[Proof of Proposition \ref{grad tau estimates}]
We have
\begin{align}\label{grad tau}
	\nonumber |\nabla\tau|^2 &= 4\delta^2\psi^2 + \delta^4|\nabla\psi|^2 + 4\delta^3(\nabla\delta\cdot\nabla\psi) + \lambda^2\left(\frac{\phi}{r_0^{\lambda}}\right)^2\left(\delta^{2\lambda-2} + \delta^{2\lambda}|\nabla\psi|^2 + 2\delta^{2\lambda-1}(\nabla\delta\cdot\nabla\psi)\right) 
	\\
	&+ \lambda\frac{\phi}{r_0^{\lambda}}\left(2\delta^{2+\lambda}|\nabla\psi|^2 + 4\delta^{\lambda}\psi + 2(1+2\psi)\delta^{1+\lambda}(\nabla\delta\cdot\nabla\psi)\right)
\end{align}
Now we observe that, for $r_0$ as in \eqref{r_0}, we have
\begin{align*}
	3\delta^2\psi^2+4\delta^3(\nabla\delta\cdot\nabla\psi) \geq \delta^2(3\psi^2-4\delta|\nabla\psi|) \geq \delta^2(3\psi^2-4r_0|\nabla\psi|) \geq 0,
\end{align*}
\begin{align*}
	2\delta^{2+\lambda}|\nabla\psi|^2 + 4\delta^{\lambda}\psi + 2(1+2\psi)\delta^{1+\lambda}(\nabla\delta\cdot\nabla\psi) \geq 2\delta^{\lambda}\left(2\psi-\delta^2|\nabla\psi|^2-(1+2\psi)\delta(\nabla\delta\cdot\nabla\psi)\right)
	\\
	\geq 2\delta^{\lambda}\left(2\psi-r_0\left(|\nabla\psi|^2+(1+2\psi)|\nabla\psi|\right)\right) \geq 0
\end{align*}
and
\begin{align*}
	\delta^{2\lambda-2} + \delta^{2\lambda}|\nabla\psi|^2 + 2\delta^{2\lambda-1}(\nabla\delta\cdot\nabla\psi) = \delta^{2\lambda-2}\left(1 + \delta^2|\nabla\psi|^2 + 2\delta(\nabla\delta\cdot\nabla\psi)\right)
	\\
	\geq \delta^{2\lambda-2}\left(1 - \delta^2|\nabla\psi|^2 - 2\delta|\nabla\psi|\right) \geq \delta^{2\lambda-2}\left(1 - r_0\left(|\nabla\psi|^2 + 2|\nabla\psi|\right)\right)\geq 0.
\end{align*}
Therefore, \eqref{grad tau est 1} immediately follows. 
\\
\indent Let us now prove \eqref{grad tau est 2}. Firstly, we observe that, thanks to Lemma \ref{lemma_psi} and to the properties of $\psi$, we get
\begin{align*}
	\delta^{2\lambda-2} &+ \delta^{2\lambda}|\nabla\psi|^2 + 2\delta^{2\lambda-1}(\nabla\delta\cdot\nabla\psi) \geq \delta^{2\lambda}\left(|\nabla\psi|^2 + \frac{2}{\delta}(\nabla\delta\cdot\nabla\psi)\right) \geq \delta^{2\lambda}\left(\varpi^2\varpi_0^2 - \frac{2\varpi\Do{\psi_1}}{r_0}\right) 
	\\
	&\geq \frac{\varpi^2\varpi_0^2}{2}\delta^{2\lambda}, 
\end{align*}
for all $x\in\tilde{\mathcal{O}}$ and for $\varpi$ as in \eqref{varpi}. Moreover,
\begin{align*}
	2\delta^{2+\lambda}|\nabla\psi|^2 + 4\delta^{\lambda}\psi + 2(1+2\psi)\delta^{1+\lambda}(\nabla\delta\cdot\nabla\psi) \geq -2(1+2\psi)\varpi\Do{\psi_1}\delta^{\lambda+1};
\end{align*}
hence
\begin{align*}
	|\nabla\tau|^2 & \geq \lambda^2\frac{\varpi^2\varpi_0^2}{2}\left(\frac{\delta}{r_0}\right)^{2\lambda}\phi^2 - 2(1+2\psi)\varpi\Do{\psi_1}R_{\Omega}\left(\frac{\delta}{r_0}\right)^{\lambda}\phi.
\end{align*}
Now, since by definition $\lambda\psi\leq\phi$,
\begin{align*}
		\lambda^2 & \frac{\varpi^2\varpi_0^2}{4}\left(\frac{\delta}{r_0}\right)^{2\lambda}\phi^2 - 2(1+2\psi)\varpi\Do{\psi_1}R_{\Omega}\left(\frac{\delta}{r_0}\right)^{\lambda}\phi 
		\\
		&= \frac{\varpi^2\varpi_0^2}{4}\left(\frac{\delta}{r_0}\right)^{2\lambda}\phi^2\left(\lambda^2 - \lambda\frac{8(1+2\psi)\varpi\Do{\psi_1}R_{\Omega}}{\varpi^2\varpi_0^2}\left(\frac{r_0}{\delta}\right)^{\lambda}\frac{1}{\phi}\right) 
		\\
		&\geq \frac{\varpi^2\varpi_0^2}{4}\left(\frac{\delta}{r_0}\right)^{2\lambda}\phi^2\left(\lambda^2 - \lambda\frac{24\psi\varpi\Do{\psi_1}R_{\Omega}}{\varpi^2\varpi_0^2}\frac{1}{\phi}\right) \geq \frac{\varpi^2\varpi_0^2}{2}\left(\frac{\delta}{r_0}\right)^{2\lambda}\phi^2\left(\lambda^2 - \frac{\lambda\psi}{\phi}\right) 
		\\
		&\geq \frac{\varpi^2\varpi_0^2}{2}\left(\frac{\delta}{r_0}\right)^{2\lambda}\phi^2\left(\lambda^2 - 1\right) 
\end{align*}
for $\varpi$ as in \eqref{varpi}. Therefore we can conclude
\begin{align*}
	|\nabla\tau|^2 & \geq \lambda^2\frac{\varpi^2\varpi_0^2}{4}\left(\frac{\delta}{r_0}\right)^{2\lambda}\phi^2,
\end{align*}
which implies \eqref{grad tau est 3}, again for $\varpi$ as in \eqref{varpi}.
\end{proof}

\subsection{Proof of the lemmas from Section \ref{carleman_proof}}
\begin{proof}[Proof of Lemma \ref{scalar product lemma}]
To simplify the presentation, we define
\begin{align*}
	&\mathrm{S}_1:=\Delta z, & &\mathrm{S}_2:=\frac{\mu}{\delta^2}z, & &\mathrm{S}_3:=\left(R\sigma_t+R^2|\nabla\sigma|^2\right)z,
	\\
	&\mathrm{A}_1:=z_t, & &\mathrm{A}_2:=2R\,\nabla\sigma\cdot\nabla z, & &\mathrm{A}_3:=R\Delta\sigma(1+\alpha)z,
\end{align*}
and we denote by $I_{i,j}$, $i,j=1,2,3$, the scalar product $\langle \mathrm{S}_i,\mathrm{A}_j\rangle$. We compute each term separately. Moreover, the computations for $I_{1,j}$ and $I_{3,j}$, $j=1,2,3$, are the same as in \cite[Lemma 2.4]{ervedoza2008control} and we will omit them here. 
\paragraph*{Computations for $I_{2,1}$}
Due to the boundary conditions \eqref{bc time}, we immediately have
\begin{align*}
	I_{2,1} = \frac{\mu}{2}\int_Q \frac{\partial_t(z^2)}{\delta^2}\,dxdt = \frac{\mu}{2}\left.\int_{\Omega}\frac{z^2}{\delta^2}\right|_0^T\,dx - \frac{\mu}{2}\int_Q z^2\partial_t\left(\frac{1}{\delta^2}\right)\,dxdt = 0.
\end{align*}
\paragraph*{Computations for $I_{2,2}$}
Applying integration by parts and \eqref{heat z bc} we have
\begin{align*}
	I_{2,2} = R\mu\int_Q\frac{1}{\delta^2}\left(\nabla\sigma\cdot\nabla(z^2)\right)\,dxdt = -R\mu\int_Q \Delta\sigma\frac{z^2}{\delta^2}\,dxdt + 2R\mu\int_Q (\nabla\delta\cdot\nabla\sigma)\frac{z^2}{\delta^3}\,dxdt.
\end{align*}
\paragraph*{Computations for $I_{2,3}$}
\begin{align*}
	I_{2,3} = R\mu\int_Q\Delta\sigma(1+\alpha)\frac{z^2}{\delta^2}\,dxdt.
\end{align*}
Identity \eqref{scalar product} follows immediately
\end{proof}

\begin{proof}[Proof of Lemma \ref{I_bd lemma}]
It is sufficient to prove that $\nabla\sigma\cdot n = 0$ for all $(x,t)\in\Gamma\times(0,T)$ and $\lambda>1$. First of all, we have
\begin{align*}
	\nabla\sigma = \theta\left(-2\delta\psi\nabla\delta - \delta^2\nabla\psi - \frac{\lambda}{r_0^{\lambda}}\left(\delta^{\lambda-1}\nabla\delta + \delta^{\lambda}\nabla\psi\right)\phi\right).
\end{align*}
\indent Moreover, because of the assumptions we made on the function $\psi$, for any $x\in\Gamma$ we have $\nabla\psi\cdot n=-|\nabla\psi|$; furthermore, it is a classical property of the distance function that $\nabla\delta\cdot n=-1$. Therefore,
\begin{align*}
	\nabla\sigma\cdot n &= \theta\left(-2\delta\psi(\nabla\delta\cdot n) + \delta^2|\nabla\psi| - \frac{\lambda}{r_0^{\lambda}}\left(\delta^{\lambda-1}\nabla\delta\cdot n - \delta^{\lambda}|\nabla\psi|\right)\phi\right)
	\\
	&= \theta\left(2\delta + \delta^2|\nabla\psi| + \frac{\lambda}{r_0^{\lambda}}\delta^{\lambda-1}\Big(1 + \delta|\nabla\psi|\Big)\,\phi\right).
\end{align*}
It is thus evident that, for any $\lambda>1$, $\nabla\sigma\cdot n = 0$ on $\Gamma\times(0,T)$.
\end{proof}

\begin{proof}[Proof of Lemma \ref{I_l lemma}]
We split $I_l$ in two parts, $I_l=I_l^1+I_l^2$, where
\begin{align}
	I_l^1 = & -2R\int_Q D^2\sigma(\nabla z,\nabla z)\,dxdt - R\int_Q \alpha\Delta\sigma|\nabla z|^2\,dxdt + 2R\mu\int_Q (\nabla\delta\cdot\nabla\sigma)\frac{z^2}{\delta^3}\,dxdt, \label{I_l 1}
	\\
	\nonumber I_l^2 = & -\frac{R}{2}\int_Q \Delta^2\sigma(1+\alpha)z^2\,dxdt + R\int_Q \left(\nabla(\Delta\sigma)\cdot\nabla\alpha\right)z^2\,dxdt + \frac{R}{2}\int_Q \Delta\sigma\Delta\alpha z^2\,dxdt
	\\
	&+ R\mu\int_Q \alpha\Delta\sigma\frac{z^2}{\delta^2}\,dxdt. \label{I_l 2}
\end{align}
Moreover, we also split $I_l^1=I_{l,\delta}^1+I_{l,\phi}^1$ where
\begin{align}
	I_{l,\delta}^1 &= -2R\int_Q D^2\sigma_{\delta}(\nabla z,\nabla z)\,dxdt - R\int_Q \alpha\Delta\sigma_{\delta}|\nabla z|^2\,dxdt + 2R\mu\int_Q (\nabla\delta\cdot\nabla\sigma_{\delta})\frac{z^2}{\delta^3}\,dxdt, \label{I_l 1 delta}
	\\
	I_{l,\phi}^1 &= -2R\int_Q D^2\sigma_{\phi}(\nabla z,\nabla z)\,dxdt - R\int_Q \alpha\Delta\sigma_{\phi}|\nabla z|^2\,dxdt + 2R\mu\int_Q (\nabla\delta\cdot\nabla\sigma_{\phi})\frac{z^2}{\delta^3}\,dxdt. \label{I_l 1 phi}
\end{align}

\paragraph*{Estimates for $I_{l,\delta}^1$}
From \eqref{deriv delta 3} and \eqref{deriv delta 4} we have 
\begin{align*}
	I_{l,\delta}^1 &= 4R\int_Q \theta\psi(\nabla\delta\cdot\nabla z)^2\,dxdt + 4R\int_Q \theta\psi\delta D^2\delta(\nabla z,\nabla z)\,dxdt 
	\\
	&+ 8R\int_Q \theta\delta(\nabla\delta\cdot\nabla z)(\nabla\psi\cdot\nabla z)\,dxdt + R\int_Q \theta\delta^2 D^2\psi(\nabla z,\nabla z)\,dxdt - R\int_Q \alpha\Delta\sigma_{\delta}|\nabla z|^2\,dxdt 
	\\
	&- 4R\mu\int_Q \theta\psi\frac{z^2}{\delta^2}\,dxdt - 2R\mu\int_Q \theta(\nabla\delta\cdot\nabla\psi)\frac{z^2}{\delta}\,dxdt.
\end{align*}
Hence
\begin{align*}
	I_{l,\delta}^1 &\geq -4R\int_Q \theta\psi(\nabla\delta\cdot\nabla z)^2\,dxdt + 4R\int_Q \theta\psi\delta D^2\delta(\nabla z,\nabla z)\,dxdt 
	\\
	&+ 8R\int_Q \theta\delta(\nabla\delta\cdot\nabla z)(\nabla\psi\cdot\nabla z)\,dxdt + R\int_Q \theta\delta^2 D^2\psi(\nabla z,\nabla z)\,dxdt - R\int_Q \alpha\Delta\sigma_{\delta}|\nabla z|^2\,dxdt 
	\\
	&- 4R\mu\int_Q \theta\psi\frac{z^2}{\delta^2}\,dxdt - 2R\mu\int_Q \theta(\nabla\delta\cdot\nabla\psi)\frac{z^2}{\delta}\,dxdt
	\\
	&\geq 4R\int_Q \theta\psi\left(|\nabla z|^2-\mu\frac{z^2}{\delta^2}\right)\,dxdt - 8R\int_Q \theta\psi|\nabla z|^2\,dxdt + 4R\int_Q \theta\psi\delta D^2\delta(\nabla z,\nabla z)\,dxdt 
	\\
	&+ 8R\int_Q \theta\delta(\nabla\delta\cdot\nabla z)(\nabla\psi\cdot\nabla z)\,dxdt + R\int_Q \theta\delta^2 D^2\psi(\nabla z,\nabla z)\,dxdt - R\int_Q \alpha\Delta\sigma_{\delta}|\nabla z|^2\,dxdt 
	\\
	&- 2R\mu\int_Q \theta(\nabla\delta\cdot\nabla\psi)\frac{z^2}{\delta}\,dxdt.
\end{align*}
Therefore,
\begin{align*}
	I_{l,\delta}^1 &\geq 4R\int_Q \theta\psi\left(|\nabla z|^2-\mu\frac{z^2}{\delta^2}\right)\,dxdt - 8R\int_Q \theta\psi|\nabla z|^2\,dxdt - 4R|D^2\delta|_{\infty}\int_Q \theta\psi|\nabla z|^2\,dxdt 
	\\
	&- 8R|D\psi|_{\infty}R_{\Omega}\int_Q \theta|\nabla z|^2\,dxdt - R|D^2\psi|_{\infty}R_{\Omega}^2\int_Q \theta|\nabla z|^2\,dxdt - R\int_Q \alpha\Delta\sigma_{\delta}|\nabla z|^2\,dxdt 
	\\
	&- 2R\mu\int_Q \theta(\nabla\delta\cdot\nabla\psi)\frac{z^2}{\delta}\,dxdt
	\\
	&\geq 4R\int_Q \theta\psi\left(|\nabla z|^2-\mu\frac{z^2}{\delta^2}\right)\,dxdt - RM_1\int_Q \theta|\nabla z|^2\,dxdt - R\int_Q \alpha\Delta\sigma_{\delta}|\nabla z|^2\,dxdt 
	\\
	&- 2R\mu\int_Q \theta(\nabla\delta\cdot\nabla\psi)\frac{z^2}{\delta}\,dxdt.
\end{align*}
where $M_1=M_1(\mu,\psi,\Omega)$ is a positive constant.
\\
\indent Next, we estimate the first term in the expression above applying the Hardy-Poincar\'e inequality \eqref{hardy 3}. First of all, by integration by parts we obtain the identities
\begin{align*}
	& \int_{\Omega} z(\nabla\psi\cdot\nabla z)\,dx = -\frac{1}{2}\int_{\Omega} z^2\Delta\psi\,dx
	\\
	\\
	& \int_{\Omega} \delta^{2-\gamma}z(\nabla\psi\cdot\nabla z)\,dx = -\frac{1}{2}\int_{\Omega} \delta^{2-\gamma}\Delta\psi z^2\,dx - \frac{2-\gamma}{2}\int_{\Omega} \delta^{1-\gamma}(\nabla\delta\cdot\nabla\psi)\,dx.
\end{align*}
Secondly, we apply \eqref{hardy 3} for $u:=z\sqrt{\psi}$ and, after integrating in time, we get
\begin{align*}
	A_4 & \int_Q \theta\psi z^2\,dxdx + \int_Q \theta\psi\left(|\nabla z|^2-\mu\frac{z^2}{\delta^2}\right)\,dxdt + \frac{1}{4}\int_Q \theta\frac{|\nabla\psi|^2}{\psi}z^2\,dxdt - \frac{1}{2}\int_Q \theta z^2\Delta\psi\,dxdt
	\\
	& \geq A_5\int_Q \theta\psi\left(\delta^{2-\gamma}|\nabla z|^2+A_1\frac{z^2}{\delta^{\gamma}}\right)\,dxdt + \frac{A_5}{4}\int_Q \theta\delta^{2-\gamma}\frac{|\nabla\psi|^2}{\psi}z^2\,dxdt - \frac{A_5}{2}\int_Q \theta\delta^{2-\gamma}z^2\Delta\psi\,dxdt 
	\\
	&- A_5\frac{2-\gamma}{2}\int_Q \theta\delta^{1-\gamma}(\nabla\delta\cdot\nabla\psi)z^2\,dxdt,
\end{align*}
where $A_5$ and $A_5$ are the constants of Proposition \ref{hardy 3 prop}. Now, for $r_0$ as in \eqref{r_0} we have
\begin{align*}
	\frac{A_5\psi}{4\delta^{\gamma}}\geq\frac{A_5}{2}(2-\gamma)\delta^{1-\gamma}|D\psi|_{\infty}, \;\;\; \forall x\in\Omega_{r_0};
\end{align*}
therefore,
\begin{align*}
	\frac{A_5}{2}\int_Q \theta\psi & \left(\delta^{2-\gamma}|\nabla z|^2+\frac{1}{2}\frac{z^2}{\delta^{\gamma}}\right)\,dxdt - A_5\frac{2-\gamma}{2}\int_Q \theta\delta^{1-\gamma}(\nabla\delta\cdot\nabla\psi)z^2\,dxdt 
	\\
	& \geq -\frac{A_5}{2}(2-\gamma)|D\psi|_{\infty}\left|\,\sup_{\delta > r_0}\delta^{1-\gamma}\,\right|\intr{\tilde{\mathcal{O}}\times (0,T)} \theta z^2\,dxdt; 
\end{align*}
combing the two expressions above, we finally obtain
\begin{align*}
	\int_Q \theta\psi\left(|\nabla z|^2-\mu^*\frac{z^2}{\delta^2}\right)\,dxdt & \geq \frac{A_5}{2}\int_Q \theta\psi\left(\delta^{2-\gamma}|\nabla z|^2+\frac{1}{2}\frac{z^2}{\delta^{\gamma}}\right)\,dxdt - A_6\int_Q \theta z^2\,dxdx, 
\end{align*}
where
\begin{align*}
	A_6:= \frac{A_5}{4}\left(R_{\Omega}^{2-\gamma}|D\psi|_{\infty}^2+2R_{\Omega}^{2-\gamma}+2(2-\gamma)|D\psi|_{\infty}\left|\sup_{\delta > r_0}\delta^{1-\gamma}\right|\,\right). 
\end{align*}
Therefore
\begin{align*}
	I_{l,\delta}^1 &\geq M_2R\int_Q \theta\psi\left(\delta^{2-\gamma}|\nabla z|^2+\frac{z^2}{\delta^{\gamma}}\right)\,dxdt - RM_1\int_Q \theta|\nabla z|^2\,dxdt - R\int_Q \alpha\Delta\sigma_{\delta}|\nabla z|^2\,dxdt 
	\\
	&- 2R\mu\int_Q \theta(\nabla\delta\cdot\nabla\psi)\frac{z^2}{\delta}\,dxdt - A_6R\int_Q \theta z^2\,dxdx.
\end{align*}
Since $\gamma>1$, for $r_0$ as in \eqref{r_0} we have 
\begin{align*}
	\frac{2|\mu||D\psi|_{\infty}}{\delta} \leq \frac{M_2}{2\delta^{\gamma}}, \;\;\; \forall x\in\Omega_{r_0};
\end{align*}
knowing this, we can finally conclude
\begin{align}\label{I_l 1 delta est}
	\nonumber I_{l,\delta}^1 &\geq B_1R\int_Q \theta\psi\left(\delta^{2-\gamma}|\nabla z|^2+\frac{z^2}{\delta^{\gamma}}\right)\,dxdt - RM_1\int_Q \theta|\nabla z|^2\,dxdt - R\int_Q \alpha\Delta\sigma_{\delta}|\nabla z|^2\,dxdt 
	\\
	&- A_6R\int_Q \theta z^2\,dxdx,
\end{align}
where $B_1:=M_2/2$.

\paragraph*{Estimates for $I_{l,\phi}^1$}
In order to get rid of the gradient terms with negative signs in \eqref{I_l 1 delta est}, we introduce the quantity
\begin{align}\label{T}
	\nonumber \mathcal{T} =& \, I_{l,\phi}^1 - R\int_Q \alpha\Delta\sigma_{\delta}|\nabla z|^2\,dxdt - RM_1\int_Q \theta|\nabla z|^2\,dxdt
	\\
	\nonumber =& -2R\int_Q D^2\sigma_{\phi}(\nabla z,\nabla z)\,dxdt - R\int_Q \alpha\Delta\sigma_{\phi}|\nabla z|^2\,dxdt + 2R\mu\int_Q (\nabla\delta\cdot\nabla\sigma_{\phi})\frac{z^2}{\delta^3}\,dxdt 
	\\
	&- R\int_Q \alpha\Delta\sigma_{\delta}|\nabla z|^2\,dxdt - RM_1\int_Q \theta|\nabla z|^2\,dxdt
\end{align}
and we need to estimate it from below. To do that, according to Propositions \ref{delta estimates} and \ref{phi estimates} we remark that
\begin{align*}
	& 2 D^2\tau_{\phi}(\nabla z,\nabla z) + \alpha\Delta\tau_{\phi}|\nabla z|^2 + \alpha\Delta\tau_{\delta}|\nabla z|^2 \geq \lambda\left(\frac{\delta}{r_0}\right)^{\lambda-2}\phi|\nabla z|^2, & \forall x\in\Omega_{r_0},
	\\
	& \left|2 D^2\tau_{\phi}(\nabla z,\nabla z) + \alpha\Delta\tau_{\phi}|\nabla z|^2 + (\alpha\Delta\tau_{\delta}-M_1)|\nabla z|^2\right| \leq M_2\lambda^2\left(\frac{\delta}{r_0}\right)^{\lambda}\phi|\nabla z|^2, & \forall x\in\omega_0,
	\\
	& 2 D^2\tau_{\phi}(\nabla z,\nabla z) + \alpha\Delta\tau_{\phi}|\nabla z|^2 + (\alpha\Delta\tau_{\delta}-M_1)|\nabla z|^2 \geq M_3 \lambda^2\left(\frac{\delta}{r_0}\right)^{\lambda}\phi|\nabla z|^2, & \forall x\in\mathcal{O},
\end{align*} 
for $\lambda$ large enough and for some positive constants $M_2$ and $M_3$ not depending on $\lambda$. On the other hand, there exists a positive constant $M_4$, again not depending on $\lambda$, such that it holds
\begin{align*}
	\left|\frac{2|\mu||(\nabla\delta\cdot\nabla\tau_{\phi})|}{\delta^3}\right| \leq M_4\lambda\left(\frac{\delta}{r_0}\right)^{\lambda-4}\phi, \;\;\;\; \forall x\in\Omega.
\end{align*}
Therefore it follows
\begin{align*}
	\mathcal{T} \geq \frac{\lambda R}{2}\intr{\Omega_{r_0}\times(0,T)} & \theta\left(\frac{\delta}{r_0}\right)^{\lambda-2}|\nabla z|^2\,dxdt - M_2\lambda^2R\intr{\omega_0\times(0,T)} \theta\left(\frac{\delta}{r_0}\right)^{\lambda}\phi|\nabla z|^2\,dxdt 
	\\
	&+ M_3\lambda^2R\intr{\mathcal{O}\times(0,T)} \theta\left(\frac{\delta}{r_0}\right)^{\lambda}\phi|\nabla z|^2\,dxdt - M_4\lambda R\int_Q \theta\left(\frac{\delta}{r_0}\right)^{\lambda-4}\phi z^2\,dxdt,
\end{align*}
for $\lambda$ large enough. Joining the two expression obtained for $I_{l,\delta}^1$ and $\mathcal{T}$ we finally have 
\begin{align}\label{I_l 1 est}
	\nonumber I_l^1 &\geq B_1R\int_Q \theta\psi\left(\delta^{2-\gamma}|\nabla z|^2+\frac{z^2}{\delta^{\gamma}}\right)\,dxdt  - A_6R\int_Q \theta z^2\,dxdx + \frac{\lambda R}{2}\intr{\Omega_{r_0}\times(0,T)} \theta\left(\frac{\delta}{r_0}\right)^{\lambda-2}|\nabla z|^2\,dxdt 
	\\
	& \nonumber - B_2\lambda^2R\intr{\omega_0\times(0,T)} \theta\left(\frac{\delta}{r_0}\right)^{\lambda}\phi|\nabla z|^2\,dxdt + B_3\lambda^2R\intr{\mathcal{O}\times(0,T)} \theta\left(\frac{\delta}{r_0}\right)^{\lambda}\phi|\nabla z|^2\,dxdt 
	\\
	&- M_5\lambda R\int_Q \theta\left(\frac{\delta}{r_0}\right)^{\lambda-4}\phi z^2\,dxdt,
\end{align}

\paragraph*{Estimates for $I_l^2$}
Using the fact that the support of $\alpha$ is located away from the origin, we note that
\begin{align*}
	\left|\alpha\frac{\Delta\tau_{\delta}}{\delta^2}\right|, \;\; \left|\alpha\frac{\Delta\tau_{\psi}}{\delta^2}\right|, \;\; |\Delta\alpha\Delta\tau_{\delta}|, \;\; |\Delta\alpha\Delta\tau_{\psi}|, \;\; |\nabla(\Delta\tau_{\delta})\cdot\nabla\alpha|, \;\; |\nabla(\Delta\tau_{\psi})\cdot\nabla\alpha|, \;\; |\Delta^2\tau_{\delta}| \leq A_{\lambda}, \;\;\; \forall x\in\Omega.
\end{align*}
Moreover, there exists a positive constant $\Upsilon$ such that
\begin{align*}
	|\Delta^2\tau_{\delta}(1+\alpha)| \leq \frac{2\Upsilon}{\delta^2}, \;\;\; \forall x\in\Omega.
\end{align*}
Hence
\begin{align*}
	I_l^2 \geq -A_{\lambda}R\int_Q \theta z^2\,dxdt - \Upsilon R\int_Q \theta|\nabla z|^2\,dxdt
\end{align*}
and, for $\lambda$ large enough, we finally have \eqref{I_l bound} with $B_{\lambda}:=A_{\lambda}+A_6+M_5\lambda\sup_{\in\Omega}\{(\delta/r_0)^{\lambda-4}\phi\}$.
\end{proof}

\begin{proof}[Proof of Lemma \ref{I_nl lemma}]
We split $I_{nl}=I_{nl,1}+I_{nl,2}$, where $I_{nl,1}$ indicates the integrals in $I_{nl}$ restricted to $\Omega_{r_0}$, while $I_{nl,2}$ are the terms in $I_{nl}$ restricted to $\tilde{\mathcal{O}}$. Moreover, if we put $\sigma=-\theta\tau$, then $I_{nl}$ can be rewritten as
\begin{align*}
	I_{nl} &= 2R^3\int_Q \theta^3 D^2\tau(\nabla\tau,\nabla\tau)z^2\,dxdt - R^3\int_Q \theta^3\alpha\Delta\tau|\nabla\tau|^2z^2\,dxdt - \frac{R^2}{2}\int_Q \theta^2\alpha^2|\Delta\tau|^2z^2\,dxdt.
\end{align*}

\paragraph*{Computations for $I_{nl,1}$} 
From \eqref{T2 est 1}, \eqref{T3 est 1} and \eqref{grad tau est 1}, for any $x\in\Omega_{r_0}$ we have 
\begin{align*}
	T_2 + T_3 &\geq \lambda^2\left(\frac{\phi}{r_0^{\lambda}}\delta^{\lambda-2}+\left(\frac{\delta}{r_0}\right)^{\lambda}\phi\right)|\nabla\tau|^2-\lambda^2\frac{\phi}{r_0^{\lambda}}|D\psi|_{\infty}^2\left(D_3\psi^2+D_4\right)\delta^{\lambda+2}
	\\
	& = \lambda^2\frac{\phi}{r_0^{\lambda}}\delta^{\lambda-2} \left(|\nabla\tau|^2 + \delta^2 |\nabla\tau|^2
	-|D\psi|_{\infty}^2\left(D_3\psi^2+D_4\right)\delta^4\right)
	\\
	& \geq \lambda^2\frac{\phi}{r_0^{\lambda}}\delta^{\lambda} \left(1 	-|D\psi|_{\infty}^2\left(D_3\psi^2+D_4\right)\delta^2\right) \geq \lambda^2\frac{\phi}{r_0^{\lambda}}\delta^{\lambda} \left(1 	-|D\psi|_{\infty}^2\left(D_3\psi^2+D_4\right)r_0^2\right) \geq 0
\end{align*}
for $r_0$ as in \eqref{r_0}. Hence, using \eqref{T1 est 1} and \eqref{grad tau est 1} we conclude
\begin{align*}
	2D^2\tau(\nabla\tau,\nabla\tau)-\alpha\Delta\tau|\nabla\tau|^2\geq \delta^2, \;\;\; \forall x\in\Omega_{r_0};
\end{align*}
as a consequence,
\begin{align*}
	I_{nl,1} \geq R^3\intr{\Omega_{r_0}\times(0,T)} \theta^3\delta^2z^2\,dxdt - \frac{R^2}{2}\intr{\Omega_{r_0}\times(0,T)} \theta^2\alpha^2|\Delta\tau|^2z^2\,dxdt.
\end{align*}
Moreover, since $\alpha$ is supported away from the boundary we also have 
\begin{align*}
	\alpha^2|\Delta\tau|^2 \leq A'_{\lambda}\delta^2, \;\;\; \forall x\in\Omega_{r_0};
\end{align*}
hence, finally, there exists $R_0=R_0(\lambda)$ large enough such that, for any $R\geq R_0$
\begin{align*}
	I_{nl,1} \geq \frac{R^3}{2}\intr{\Omega_{r_0}\times(0,T)} \theta^3\delta^2z^2\,dxdt.
\end{align*}

\paragraph*{Computations for $I_{nl,2}$}
According to Propositions \ref{T1 estimates}, \ref{T2 estimates} and \ref{T3 estimates} and to \eqref{grad tau est 2}, for all $x\in\mathcal{O}$ we have 
\begin{align*}
	2D^2\tau(\nabla\tau,\nabla\tau)-\alpha\Delta\tau|\nabla\tau|^2 \geq G_1\lambda^2\left(\frac{\delta}{r_0}\right)^{\lambda}\phi|\nabla\tau|^2 \geq G_1\lambda^4\left(\frac{\delta}{r_0}\right)^{3\lambda}\phi^3.
\end{align*}
In addition, it holds
\begin{align*}
	& \alpha^2|\Delta\tau|^2 \leq G_2\lambda^4\left(\frac{\delta}{r_0}\right)^{2\lambda}\phi^2, & & \forall x\in\tilde{\mathcal{O}},
	\\
	& \left|2D^2\tau(\nabla\tau,\nabla\tau)-\alpha\Delta\tau|\nabla\tau|^2\right| \leq G_3\lambda^2\left(\frac{\delta}{r_0}\right)^{\lambda}\phi|\nabla\tau|^2 \leq G_4\lambda^4\left(\frac{\delta}{r_0}\right)^{3\lambda}\phi^3, & & \forall x\in\omega_0.
\end{align*}
\indent The previous inequalities follows from \eqref{delta tau}, \eqref{d tau} and \eqref{grad tau est 3}; the constants $G_1$, $G_2$, $G_3$ and $G_4$ are all positive and independent on $\lambda$. Therefore we obtain
\begin{align*}
	I_{nl,2} &\geq G_1\lambda^4R^3\intr{\mathcal{O}\times(0,T)} \theta^3\left(\frac{\delta}{r_0}\right)^{3\lambda}\phi^3z^2\,dxdt - G_4\lambda^4R^3\intr{\omega_0\times(0,T)} \theta^3\left(\frac{\delta}{r_0}\right)^{3\lambda}\phi^3z^2\,dxdt 
	\\
	&- \frac{G_2}{2}\lambda^4R^2\intr{\tilde{\mathcal{O}}\times(0,T)} \theta^2\left(\frac{\delta}{r_0}\right)^{2\lambda}\phi^2\,dxdt.
\end{align*}
\indent Joining now the two expressions we get for $I_{nl.1}$ and $I_{nl,2}$, we finally obtain that there exists $R_0=R_0(\lambda)$ large enough such that for $R\geq R_0$ 
\begin{align*}
	I_{nl} &\geq \frac{R^3}{2}\intr{\Omega_{r_0}\times(0,T)} \theta^3\delta^2z^2\,dxdt + G_5\lambda^4R^3\intr{\mathcal{O}\times(0,T)} \theta^3\left(\frac{\delta}{r_0}\right)^{3\lambda}\phi^3z^2\,dxdt 
	\\
	&- G_6\lambda^4R^3\intr{\omega_0\times(0,T)} \theta^3\left(\frac{\delta}{r_0}\right)^{3\lambda}\phi^3z^2\,dxdt, 
\end{align*}
where $G_5:=G_1/2$ and $G_6:=G_2/2+G_4$.
\end{proof}

\begin{proof}[Proof of Lemma \ref{I_r lemma}]
According to the expression of $\theta$, there exists a constant $\varsigma>0$ such that
\begin{align*}
	|\theta\theta_t|\leq\varsigma\theta^{\,3}, \;\;\;\; |\theta_{tt}|\leq\varsigma\theta^{\,5/3};
\end{align*}
on the other hand, from the definition of $\sigma$ we obtain
\begin{align}\label{sigma time est}
	\nonumber & |\Delta\sigma| \leq E_{\lambda}\theta, \;\;\; |\sigma_t| \leq E_{\lambda}\theta_t, & & \forall x\in\Omega,
	\\
	\nonumber & \partial_t\left(|\nabla\sigma|^2\right) \leq E_{\lambda}\theta\theta_t\delta^2, & & \forall x\in\Omega_{r_0},
	\\
	& \partial_t\left(|\nabla\sigma|^2\right) \leq E_{\lambda}\theta\theta_t\left(\frac{\delta}{r_0}\right)^{2\lambda}\phi^2 & & \forall x\in\tilde{\mathcal{O}},
\end{align}
for some positive constant $E_{\lambda}$ big enough. 
\\
Since $\alpha$ is supported away from the boundary, we can write 
\begin{align*}
	R^2\int_Q \left|\alpha\sigma_t\Delta\sigma z^2\right|\,dxdt \leq \frac{\varsigma E_{\lambda}^2}{r_0^2}R^2\intr{\Omega_{r_0}\times(0,T)} \theta^{\,3}\delta^2z^2\,dxdt + \varsigma E_{\lambda}^2R^2\intr{\tilde{\mathcal{O}}\times(0,T)} \theta^{\,3}\delta^2z^2\,dxdt.
\end{align*}
Furthermore, from \eqref{sigma time est} we obtain
\begin{align*}
	R^2\left|\int_Q \partial_t\left(|\nabla\sigma|^2\right)z^2\,dxdt\,\right| \leq \varsigma E_{\lambda}R^2 \intr{\Omega_{r_0}\times(0,T)} \theta^{\,3}\delta^2z^2\,dxdt + \varsigma E_{\lambda}R^2\intr{\tilde{\mathcal{O}}\times(0,T)} \theta^{\,3}\left(\frac{\delta}{r_0}\right)^{2\lambda}\phi^2z^2\,dxdt.
\end{align*}
Now we define
\begin{align*}
	\Theta:= -\frac{R}{2}\int_Q \sigma_{tt}z^2\,dxdt - B_{\lambda}R\int_Q \theta z^2\,dxdt,
\end{align*}
where $B_{\lambda}$ is the same introduced in Lemma \ref{I_l lemma}. It is straightforward that there exists a positive constant $F_{\lambda}$ such that
\begin{align*}
	|\Theta| \leq 2F_{\lambda}R\int_Q \theta^{\,5/3}z^2\,dxdt.
\end{align*}
Next, for $1<q,\,q'<\infty$ such that $1/q+1/q'=1$ and $\ell>0$ we can write 
\begin{align*}
	\int_Q \theta^{\,5/3}z^2\,dxdt = \int_Q \left(\ell\theta^{\,5/3-1/q'}\delta^{1/q'}z^{2/q}\right)\left(\frac{1}{\ell}\theta^{1/q'}\delta^{-1/q'}z^{2/q'}\right)\,dxdt;
\end{align*}
choosing $q=3$ and $q'=3/2$ in the previous expression, and using Young's inequality, we obtain
\begin{align*}
	\int_Q \theta^{\,5/3}z^2\,dxdt \leq \frac{\ell^3}{3}\int_Q \theta^3\delta^2z^2\,dxdt + \frac{2R_{\Omega}^{\gamma-1}}{3\ell^{\,3/2}}\int_Q \theta\frac{z^2}{\delta^{\gamma}}\,dxdt,
\end{align*}
for some positive parameter $\gamma\in (1,2)$. Therefore we have
\begin{align*}
	|\Theta| \leq 2F_{\lambda}R\left( \frac{\ell^3}{3}\int_Q \theta^3\delta^2z^2\,dxdt + \frac{2R_{\Omega}^{\gamma-1}}{3\ell^{\,3/2}}\int_Q \theta\frac{z^2}{\delta^{\gamma}}\,dxdt\right).
\end{align*}
Consequently, it follows that
\begin{align*}
	|I_r| &\leq G_{\lambda}\left(R^2\intr{\Omega_{r_0}\times(0,T)} \theta^3\delta^2z^2\,dxdt \right.
	\\
	&\left. + \ell^3 R\int_Q \theta^3\delta^2z^2\,dxdt + \frac{R}{\ell^{\,3/2}}\int_Q \theta\frac{z^2}{\delta^{\gamma}}\,dxdt + R^2\intr{\tilde{\mathcal{O}}} \theta^3\left(\frac{\delta}{r_0}\right)^{2\lambda}z^2\,dxdt\right),
\end{align*}
for some new constant $G_{\lambda}>0$. Take now $\ell$ such that $G_{\lambda}/\ell^{\,3/2}=B_1/2$; then there exists $R_0=R_0(\lambda)$ such that for any $R\geq R_0$ \eqref{I_r bound} holds. 
\\
\indent We conclude pointing out that, if we choose an exponent $k<3$ for the function $\theta$ in the definition of our weight $\sigma$ (see Section \ref{carleman sec}), it is straightforward to check that some of the passages in the computations above are not true anymore and there are terms in the expression $I_r$ that we are not able to handle. Therefore, the value $k=3$ turns out to be sharp for obtaining our Carleman inequality.
\end{proof}

\section{Proof of the Propositions of Section \ref{h-p_ineq}}

\begin{proof}[Proof of Proposition \ref{hardy 2 prop}]
We split the proof in two parts: firstly, we derive \eqref{hardy 2} in $\Omega_{r_0}$ and, in a second moment, we extend the result to the whole $\Omega$.
\paragraph*{Step 1. inequality on $\Omega_{r_0}$}
Let us consider a smooth function $\phi>0$ which satisfies 
\begin{align}\label{phi condition}
	& -\Delta\phi \geq \mu\frac{\phi}{\delta^2}+\phi^p, \;\;\; \forall p\in\left[1,\frac{N-k+2}{N-k-2}\right),
\end{align}
for $k\in(1,N-2)$. According to \citep{fall2013nonexistence}, for $\delta<1$ the function 
\begin{align}\label{phi expr}
	\delta^{\,-A_k^{1/2}\left(1-\delta^{1/2}\right)}\left(1+\frac{1}{\log\delta}\right), \;\;\; A_k:=\left(\frac{N-k-2}{2}\right)^2
\end{align}
satisfies \eqref{phi condition}. Hence, for any $x\in\Omega_{r_0}$ with $r_0\leq 1$ we define $v:=\phi z$ for $z\in C_0^{\infty}(\Omega_{r_0})$; in particular, $v\in C_0^{\infty}(\Omega_{r_0})$ and
\begin{align*}
	|\nabla v|^2=\phi^2|\nabla z|^2+z^2|\nabla\phi|^2+\frac{1}{2}\nabla(\phi^2)\cdot\nabla(z^2).
\end{align*}
By applying integration by parts, it is simply a matter of computations to show
\begin{align*}
	\int_{\Omega_{r_0}} |\nabla v|^2\,dx = \int_{\Omega_{r_0}} \phi^2|\nabla z|^2\,dx - \int_{\Omega_{r_0}} \frac{\Delta\phi}{\phi}v^2\,dx
\end{align*}
and
\begin{align*}
	\frac{1}{2}\int_{\Omega_{r_0}} \delta^{2-\gamma}\nabla(\phi^2)\cdot\nabla(z^2)\,dx &= -(2-\gamma)\int_{\Omega_{r_0}} \delta^{1-\gamma}\frac{\nabla\phi\cdot\nabla\delta}{\phi}v^2\,dx - \int_{\Omega_{r_0}} \delta^{2-\gamma}\frac{\Delta\phi}{\phi}v^2\,dx 
	\\
	&- \int_{\Omega_{r_0}}\delta^{2-\gamma}|\nabla\phi|^2z^2\,dx.
\end{align*}
The two identities above implies 
\begin{align*}
	\int_{\Omega_{r_0}} \delta^{2-\gamma}\phi^2|\nabla z|^2\,dx &\leq R_{\Omega}^{2-\gamma}\int_{\Omega_{r_0}} \phi^2|\nabla z|^2\,dx = R_{\Omega}^{2-\gamma}\int_{\Omega_{r_0}} \left(|\nabla v|^2 + \frac{\Delta\phi}{\phi}v^2\right)\,dx 
	\\
	& \leq R_{\Omega}^{2-\gamma}\int_{\Omega_{r_0}} \left(|\nabla v|^2 - \mu\frac{v^2}{\delta^2}\right)\,dx - R_{\Omega}^{2-\gamma}\int_{\Omega_{r_0}} \phi^{p-1}v^2\,dx
\end{align*}
and 
\begin{align*}
	\int_{\Omega_{r_0}} \delta^{2-\gamma}|\nabla v|^2\,dx &= \int_{\Omega_{r_0}} \delta^{2-\gamma}\phi^2|\nabla z|^2\,dx - (2-\gamma)\int_{\Omega_{r_0}} \delta^{1-\gamma}\frac{\nabla\phi\cdot\nabla\delta}{\phi}v^2\,dx - \int_{\Omega_{r_0}} \delta^{2-\gamma}\frac{\Delta\phi}{\phi}v^2\,dx; 
\end{align*}
hence
\begin{align*}
	\int_{\Omega_{r_0}} \delta^{2-\gamma}|\nabla v|^2\,dx &\leq R_{\Omega}^{2-\gamma}\int_{\Omega_{r_0}} \left(|\nabla v|^2 - \mu\frac{v^2}{\delta^2}\right)\,dx - R_{\Omega}^{2-\gamma}\int_{\Omega_{r_0}} \phi^{p-1}v^2\,dx + \mu\int_{\Omega_{r_0}} \delta^{2-\gamma}\frac{v^2}{\delta^2}\,dx
	\\
	&+ \int_{\Omega_{r_0}} \delta^{2-\gamma}\phi^{p-1}v^2\,dx - (2-\gamma)\int_{\Omega_{r_0}} \delta^{1-\gamma}\frac{\nabla\phi\cdot\nabla\delta}{\phi}v^2\,dx.
\end{align*}
Now, again by integration by parts we have
\begin{align*}
	- (2-\gamma)\int_{\Omega_{r_0}} &\delta^{1-\gamma}\frac{\nabla\phi\cdot\nabla\delta}{\phi}v^2\,dx 
	\\
	&=  \int_{\Omega_{r_0}} \delta^{2-\gamma}\frac{\Delta\phi}{\phi}v^2\,dx - \int_{\Omega_{r_0}} \frac{\delta^{2-\gamma}}{\phi^2}|\nabla\phi|^2v^2\,dx + 2\int_{\Omega_{r_0}} \delta^{2-\gamma}\frac{\nabla\phi\cdot\nabla v}{\phi}v\,dx
	\\
	& \leq -\mu\int_{\Omega_{r_0}} \delta^{2-\gamma}\frac{v^2}{\delta^2}\,dx - \int_{\Omega_{r_0}} \delta^{2-\gamma}\phi^{p-1}v^2\,dx + 2\int_{\Omega_{r_0}} \delta^{2-\gamma}\frac{\nabla\phi\cdot\nabla v}{\phi}v\,dx;
\end{align*}
therefore
\begin{align*}
	\int_{\Omega_{r_0}} \delta^{2-\gamma}&|\nabla v|^2\,dx 
	\\
	&\leq R_{\Omega}^{2-\gamma}\int_{\Omega_{r_0}} \left(|\nabla v|^2 - \mu\frac{v^2}{\delta^2}\right)\,dx - R_{\Omega}^{2-\gamma}\int_{\Omega_{r_0}} \phi^{p-1}v^2\,dx + 2\int_{\Omega_{r_0}} \delta^{2-\gamma}\frac{\nabla\phi\cdot\nabla v}{\phi}v\,dx
	\\
	&\leq R_{\Omega}^{2-\gamma}\int_{\Omega_{r_0}} \left(|\nabla v|^2 - \mu\frac{v^2}{\delta^2}\right)\,dx + P_1 \int_{\Omega_{r_0}} \phi^{p-1}v^2\,dx + 2\int_{\Omega_{r_0}} \delta^{2-\gamma}\frac{\nabla\phi\cdot\nabla v}{\phi}v\,dx
	\\
	&\leq R_{\Omega}^{2-\gamma}\int_{\Omega_{r_0}} \left(|\nabla v|^2 - \mu\frac{v^2}{\delta^2}\right)\,dx + P_2 \int_{\Omega_{r_0}} v^2\,dx + 2\int_{\Omega_{r_0}} \delta^{2-\gamma}\frac{\nabla\phi\cdot\nabla v}{\phi}v\,dx.
\end{align*}
By definition of $\phi$ we have 
\begin{align*}
	\frac{\nabla\phi\cdot\nabla v}{\phi} = \left(1+\frac{1}{\log\delta}\right)^{-1}\left(\frac{A_k^{1/2}}{2}\frac{\log\delta}{\delta^{1/2}}-A_k^{1/2}\frac{1-\delta^{1/2}}{\delta}-\frac{1}{\delta\log^2\delta}\right)(\nabla\delta\cdot\nabla v);
\end{align*}
plugging this expression  in the inequality above we immediately get
\begin{align*}
	\int_{\Omega_{r_0}} \delta^{2-\gamma}|\nabla v|^2dx &\leq R_{\Omega}^{2-\gamma}\int_{\Omega_{r_0}} \left(|\nabla v|^2 - \mu\frac{v^2}{\delta^2}\right)dx + P_2\int_{\Omega_{r_0}} v^2dx + P_3\int_{\Omega_{r_0}} \delta^{2-\gamma}\frac{\log\delta}{\delta^{1/2}}(\nabla\delta\cdot\nabla v)v\,dx
\end{align*}
with 
\begin{align*}
	P_3:=A_k^{1/2}\left|\sup_{x\in\Omega_{r_0}} \left(1+\frac{1}{\log\delta}\right)^{-1}\,\right|.
\end{align*}
Now, using another time integration by parts, and since $\log\delta<\delta^{3/2}$, we finally obtain 
\begin{align*}
	\int_{\Omega_{r_0}} \delta^{2-\gamma}|\nabla v|^2\,dx &\leq R_{\Omega}^{2-\gamma}\int_{\Omega_{r_0}} \left(|\nabla v|^2 - \mu\frac{v^2}{\delta^2}\right)\,dx + P_2\int_{\Omega_{r_0}} v^2\,dx + P_3\int_{\Omega_{r_0}} \delta^{3-\gamma}(\nabla\delta\cdot\nabla(v^2))\,dx
	\\
	&\leq R_{\Omega}^{2-\gamma}\int_{\Omega_{r_0}} \left(|\nabla v|^2 - \mu\frac{v^2}{\delta^2}\right)\,dx + A_2\int_{\Omega_{r_0}} v^2\,dx, 
\end{align*}
where 
\begin{align*}
	A_2:= P_2+P_3\left[R_{\Omega}^{2-\gamma}(3-\gamma)+R_{\Omega}^{3-\gamma}|\Delta\delta|\right].
\end{align*}

\paragraph*{Step 2. inequality on $\Omega$}
We apply a cut-off argument to recover the validity of the inequality on the whole $\Omega$. More in details, we consider a function $\psi\in C_0^{\infty}(\RR^N)$ such that 
\begin{align*}
	\psi(x)=\left\{\begin{array}{ll}
		1, & \forall x\in\Omega_{r_0/2},
		\\
		0, & \forall x\in\Omega\setminus\Omega_{r_0}
	\end{array}\right.
\end{align*}
and we split $v\in C_0^{\infty}(\Omega)$ as $v=\psi v + (1-\psi)v:=v_1+v_2$. Thus, we get
\begin{align*}
	\int_{\Omega} \delta^{2-\gamma}|\nabla v|^2\,dx &= \int_{\Omega_{r_0}} \delta^{2-\gamma}|\nabla v_1|^2\,dx + \intr{\Omega\setminus\Omega_{r_0/2}} \delta^{2-\gamma}|\nabla v_2|^2\,dx + 2\intr{\Omega_{r_0}\setminus\Omega_{r_0/2}} \delta^{2-\gamma}(\nabla v_1\cdot\nabla v_2)\,dx. 
\end{align*}
Applying \eqref{hardy 2} to the previous identity we obtain
\begin{align*}
	\int_{\Omega} \delta^{2-\gamma}|\nabla v|^2\,dx &\leq R_{\Omega}^{2-\gamma}\left(\int_{\Omega} |\nabla v|^2\,dx - \mu\int_{\Omega_{r_0}}\frac{v^2}{\delta^2}\,dx\right) - \intr{\Omega_{r_0}\setminus\Omega_{r_0/2}} 2\left(R_{\Omega}^{2-\gamma}-\delta^{2-\gamma}\right)(\nabla v_1\cdot\nabla v_2)\,dx
	\\
	&+ J_1\int_{\Omega} v^2\,dx.
\end{align*}
\indent As shown in \cite[Lemma 5.1]{cazacu2012schrodinger}, for a smooth function $q: C^{\infty}(\Omega)\to\RR$ which is bounded and non-negative, there exists a constant $C>0$ depending on $\Omega$ and $q$ such that it holds
\begin{align}\label{weight ineq}
	\int_{\Omega} q(x)(\nabla v_1\cdot\nabla v_2)\,dx \geq -C\int_{\Omega} v^2\,dx;
\end{align}
hence, considering \eqref{weight ineq} with 
\begin{align*}
	q=\left. 2\left(R_{\Omega}^{2-\gamma}-\delta^{2-\gamma}\right)\,\right|_{\Omega_{r_0}\setminus\Omega_{r_0/2}}
\end{align*}
we get
\begin{align}\label{hardy 2 prel}
	\int_{\Omega} \delta^{2-\gamma}|\nabla v|^2\,dx &\leq R_{\Omega}^{2-\gamma}\left(\int_{\Omega} |\nabla v|^2\,dx - \mu\int_{\Omega_{r_0}}\frac{v^2}{\delta^2}\,dx\right) + J_2\int_{\Omega} v^2\,dx.
\end{align}
On the other hand we have
\begin{align*}
	\int_{\Omega_{r_0}}\frac{v^2}{\delta^2}\,dx \geq \int_{\Omega}\frac{v^2}{\delta^2}\,dx - J_3\int_{\Omega} v^2\,dx.
\end{align*}
Plugging this last inequality in \eqref{hardy 2 prel}, we finally obtain \eqref{hardy 2}.
\end{proof}
% acknowledgemnts
\section*{Acknowledgements} The authors wish to thank Prof. Mahamadi Warma (University of Puerto Rico) and Dr. Cristian Cazacu (University Politehnica of Bucharest) for fruitful discussions that led to this work. 
% biblio
\section*{Bibliography}
\bibliography{biblio}

\end{document}